\documentclass[twoside,11pt]{article}

\usepackage{jmlr2e}

\usepackage{amssymb,amsmath,color}
\usepackage{dsfont}
\usepackage{graphicx}
\usepackage{algorithm}
\usepackage{algpseudocode}

\usepackage{amsthm}
\usepackage{amsmath,amssymb,amsthm,amsfonts}
\usepackage{nicefrac}
\usepackage[T1]{fontenc}
\usepackage{natbib} 
\usepackage{hyperref}
\hypersetup{colorlinks=true,linkcolor=blue}

\usepackage[utf8]{inputenc} 
\usepackage{url}            
\usepackage{booktabs}       
\usepackage{microtype}      

\newcommand\cut[1]{}

\newcommand{\squishlist}{
   \begin{list}{$\bullet$}
    { \setlength{\itemsep}{0pt}      \setlength{\parsep}{3pt}
      \setlength{\topsep}{3pt}       \setlength{\partopsep}{0pt}
      \setlength{\leftmargin}{1.5em} \setlength{\labelwidth}{1em}
      \setlength{\labelsep}{0.5em} } }

\newcommand{\squishlisttwo}{
   \begin{list}{$\bullet$}
    { \setlength{\itemsep}{0pt}    \setlength{\parsep}{0pt}
      \setlength{\topsep}{0pt}     \setlength{\partopsep}{0pt}
      \setlength{\leftmargin}{2em} \setlength{\labelwidth}{1.5em}
      \setlength{\labelsep}{0.5em} } }

\newcommand{\squishend}{
    \end{list}  }

{}
\newtheorem{thm}{Theorem}{}
{}
{}
{}

\newcommand{\be}{\begin{equation}}
\newcommand{\ee}{\end{equation}}
\newcommand{\bea}{\begin{eqnarray}}
\newcommand{\eea}{\end{eqnarray}}
\newcommand{\beaa}{\begin{eqnarray*}}
\newcommand{\eeaa}{\end{eqnarray*}}

\DeclareMathOperator*{\argmin}{arg\,min}
\DeclareMathOperator*{\argmax}{arg\,max}

\newtheorem{asm}{Assumption}[section]
\newtheorem{cor}[thm]{Corollary}
\newtheorem{rmk}{Remark}[section]

\ShortHeadings{Convergence Rates of Variational Inference in Sparse Deep Learning}{Ch\'erief-Abdellatif}
\firstpageno{1}

\begin{document}

\title{Convergence Rates of Variational Inference in \\
Sparse Deep Learning}

\author{\name Badr-Eddine Ch\'erief-Abdellatif \email badr.eddine.cherief.abdellatif@ensae.fr \\
       \addr CREST, ENSAE, Institut Polytechnique de Paris}

\editor{}

\maketitle

\begin{abstract}
Variational inference is becoming more and more popular for approximating intractable posterior distributions in Bayesian statistics and machine learning. Meanwhile, a few recent works have provided theoretical justification and new insights on deep neural networks for estimating smooth functions in usual settings such as nonparametric regression. In this paper, we show that variational inference for sparse deep learning retains the same generalization properties than exact Bayesian inference. In particular, we highlight the connection between estimation and approximation theories via the classical bias-variance trade-off and show that it leads to near-minimax rates of convergence for H\"older smooth functions. Additionally, we show that the model selection framework over the neural network architecture via ELBO maximization does not overfit and adaptively achieves the optimal rate of convergence.
\end{abstract}

\begin{keywords}
  Variational Inference, Neural Networks, Deep Learning, Generalization
\end{keywords}

\section{Introduction}

Deep learning (DL) is a field of machine learning that aims to model data using complex architectures combining several nonlinear transformations with hundreds of parameters called Deep Neural Networks (DNN) \citep{deeplearningbooklecun,Goodfellow-et-al-2016}. Although generalization theory that explains why DL generalizes so well is still an open problem, it is widely acknowledged that it mainly takes advantage of large datasets containing millions of samples and a huge computing power coming from clusters of graphics processing units. Very popular architectures for deep neural networks such as the multilayer perceptron, the convolutional neural network \citep{CNNreference98lecun}, the recurrent neural network \citep{RumelhartRNN1986} or the generative adversarial network \citep{GAN-Goodfellow} have shown impressive results and have enabled to perform better than humans in various important areas in artificial intelligence such as image recognition, game playing, machine translation, computer vision or natural language processing, to name a few prominent examples. An outstanding example is AlphaGo \citep{alphago}, an artificial intelligence developed by Google that learned to play the game of Go using deep learning techniques and even defeated the world champion in 2016.

The Bayesian approach, leading to popular methods such as Hidden Markov Models \citep{HiddenMarkovModel1966} and Particle Filtering \citep{DoucetParticleFiltering}, provides a natural way to model uncertainty. Some prior distribution is put over the space of parameters and represents the prior belief as to which parameters are likely to have generated the data before any datapoint is observed. Then this prior distribution is updated using the Bayes rule when new data arrive in order to capture the more likely parameters given the observations. Unfortunately, exact Bayesian inference is computationally challenging for complex models as the normalizing constant of the posterior distribution is often intractable. In such cases, approximate inference methods such as variational inference (VI) \citep{VIJordan1999} and expectation propagation \citep{MinkaEP} are popular to overcome intractability in Bayesian modeling. The idea of VI is to minimize the Kullback-Leibler (KL) divergence with respect to the posterior given a set of tractable distributions, which is also equivalent to maximizing a numerical criterion called the Evidence Lower Bound (ELBO). Recent advances of VI have shown great performance in practice and have been applied to many machine learning problems \citep{hoffman2013stochastic, kingma2013auto}.

The Bayesian approach to learning in neural networks has a long history. Bayesian Neural Networks (BNN) have been first proposed in the 90s and widely studied since then \citep{McKay1992BNN,NealPhD}. They offer a probabilistic interpretation and a measure of uncertainty for DL models. They are more robust to overfitting than classical neural networks and still achieve great performance even on small datasets. A prior distribution is put on the parameters of the network, namely the weight matrices and the bias vectors, for instance a Gaussian or a uniform distribution, and Bayesian inference is done through the likelihood specification. Nevertheless, state-of-the-art neural networks may contain millions of parameters and the form of a neural network is not adapted to exact integration, which makes the posterior distribution be intractable in practice. Modern approximate inference mainly relies on VI, with sometimes a flavor of sampling techniques. A lot of recent papers have investigated variational inference for DNNs \citep{HintonVanCamp1993,GravesVI2011,Blundell2015} to fit an approximate posterior that maximizes the evidence lower bound. For instance, \cite{Blundell2015} introduced Bayes by Backprop, one of the most famous techniques of VI applied to neural networks, which derives a fully factorized Gaussian approximation to the posterior: using the reparameterization trick \citep{OpperArchambeau2009}, the gradients of ELBO towards parameters of the Gaussian approximation can be computed by backpropagation, and then be used for updates. Another point of interest in DNNs is the choice of the prior. \cite{Blundell2015} introduced a mixture of Gaussians prior on the weights, with one mixture tightly concentrated around zero, imitating the sparsity-inducing spike-and-slab prior. This offers a Bayesian alternative to the dropout regularization procedure \citep{DropoutSrivastava2014} which injects sparsity in the network by switching off randomly some of the weights of the network. This idea goes back to David MacKay who discussed in his thesis the possibility of choosing a spike-and-slab prior over the weights of the neural network \citep{McKayPhD}. More recently, \cite{Rockova2018} introduced Spike-and-Slab Deep Learning (SS-DL), a fully Bayesian alternative to dropout for improving generalizability of deep ReLU networks.

\subsection{Related work}

Although deep learning is extremely popular, the study of generalization properties of DNNs is still an open problem. Some works have been conducted in order to investigate the theoretical properties of neural networks from different points of view. The literature developed in the past decades can be shared in three parts. First, the approximation theory wonders how well a function can be approximated by neural networks. The first studies were mostly conducted to obtain approximation guarantees for shallow neural nets with a single hidden layer \citep{Cybenko89ShallowNN,Barron93NN}. Since then, modern research has focused on the expressive power of depth and extended the previous results to deep neural networks with a larger number of layers \citep{BengioDelalleau2011,Yarotsky2017,Petersen2018Approximation,DNNApproximationTheory2019}. Indeed, even though the universal approximation theorem \citep{Cybenko89ShallowNN} states that a shallow neural network containing a finite number of neurons can approximate any continuous function on compact sets under mild assumptions on the activation function, recent advances showed that a shallow network requires exponentially many neurons in terms of the dimension to represent a monomial function, whereas linearly many neurons are sufficient for a deep network \citep{RolnickTegmark2018PowerOfDepth}. Second, as the objective function in deep learning is known to be nonconvex, the optimization community has discussed the landscape of the objective as well as the dynamics of some learning algorithms such as Stochastic Gradient Descent (SGD) \citep{BaldiOptiNNPCA1989,Amari2000PlateausDNN,SoudryLocalMinima2016,KawaguchiNoPoorLocalMinima2016,Kawaguchi2018EffectOfDepth,Nguyen2019LandscapeDNN,ConvergenceDeepLearningAllenZhu,GlobalMinimaDNNDu}. Finally, the statistical learning community has investigated generalization properties of DNNs, see \cite{Barron94estimation,ZhangUnderstandingDL2017,SchmidtHieberDNN,Suzuki18DNNkerels,Imaizumi19DNN,suzuki2019adaptivity}. In particular, \cite{SchmidtHieberDNN} and \cite{suzuki2019adaptivity} showed that estimators in nonparametric regression based on sparsely connected DNNs with ReLU activation function and wisely chosen architecture achieve the minimax estimation rates (up to logarithmic factors) under classical smoothness assumptions on the regression function. In the same time, \cite{Bartlett2017MarginBoundsNNs} and \cite{Srebro2018PACMarginBoundsNNs} respectively used Rademacher complexity and covering number, and PAC-Bayes theory to get spectrally-normalized margin bounds for deep ReLU networks. More recently, \cite{Imaizumi19DNN} and \cite{Suzuki2019Superiority} showed the superiority of DNNs over linear operators in some situations when DNNs achieve the minimax rate of convergence while alternative methods fail. From a Bayesian point of view, \cite{Rockova2018} and \cite{Suzuki18DNNkerels} studied the concentration of the posterior distribution while \cite{Vladimirova2019PriorsBNNsUnits} investigated the regularization effect of prior distributions at the level of the units.

Such as for generalization properties of DNNs, only little attention has been put in the literature towards the theoretical properties of VI until recently. \cite{alquier2016properties} studied generalization properties of variational approximations of Gibbs distributions in machine learning for bounded loss functions. \cite{Tempered,Chicago,RoniKhardonExcessRiskBounds2017,Plage,cherief2018consistency,cherief2018consistency2,Jaiswal2019RiskSensitiveVB} extended the previous guarantees to more general statistical models and studied the concentration of variational approximations of the posterior distribution, while \cite{wang2018frequentist} provided Bernstein-von-Mises' theorems for variational approximations in parametric models. \cite{HugginsBroderick2018PracticalBounds,UBVICampbell2019,jaiswal2019asymptotic} discussed theoretical properties of variational inference algorithms based on various divergences (respectively Wasserstein and Hellinger distances, and R\'enyi divergence). More recently, \cite{cheriefAlquierKhan2019} presented generalization bounds for online variational inference. All these works show that under mild conditions, the variational approximation is consistent and achieves the same rate of convergence than the Bayesian posterior distribution it approximates. Note that \cite{Tempered,Plage,cherief2018consistency,cherief2018consistency2} restricted their studies to tempered versions of the posterior distribution where the likelihood is raised to an $\alpha$-power ($\alpha<1$) as it is known to require less stringent assumptions to obtain consistency and to be robust to misspecification, see respectively \cite{bhattacharya2016bayesian} and \cite{grunwaldmisspecifiation}. Nevertheless, some questions remain unanswered, as the theoretical study of generalization of variational inference for deep neural networks.

\subsection{Contributions}

This paper aims at filling the gap between theory and practice when using variational approximations for tempered Bayesian Deep Neural Networks. To the best of our knowledge, this is the first paper to present theoretical generalization error bounds of variational inference for Bayesian deep learning. Inspired by the related literature, our work is motivated by the following questions:
\begin{itemize}
    \item Do consistency of Bayesian DNNs still hold when an approximation is used instead of the exact posterior distribution, and can we obtain the same rates of convergence than those obtained for the regular posterior distribution and frequentist estimators ?
    \item Is it possible to obtain a nonasymptotic generalization error bound that holds for (almost) any generating distribution function and that gives a general formula ?
    \item What about the consistency of numerical algorithms used to compute these variational approximations ?
    \item Can we obtain new insights on the structure of the networks ?
\end{itemize}

The main contribution of this paper, a nonasymptotic generalization error bound for variational inference in sparse DL in the nonparametric regression framework, answers the first two questions. This generalization result is similar to theoretical inequalities in the seminal works of \cite{Suzuki18DNNkerels,Imaizumi19DNN,Rockova2018} on generalization properties of deep neural networks, and is inspired by the general literature on the consistency of variational approximations \citep{Tempered,Plage}. In particular, it states that under the same conditions, sparse variational approximations of posterior distributions of deep neural networks are consistent at the same rate of convergence than the exact posterior.

It also raises the question of finding a relevant general definition of consistency that can be used to provide theoretical properties for the exact Bayesian DNNs distribution and their variational approximations. Indeed, a classical criterion used to assess frequentist guarantees for Bayesian estimators is the concentration of the posterior (to the true distribution) which is defined as the asymptotic concentration of the Bayesian estimator to the true distribution \citep{ghosal2000convergence}. Nevertheless, posterior concentration to the true distribution only applies when the model is well specified, or at least when the model contains distributions in the neighborhood of the true distribution, which is problematic for misspecified models e.g. when the neural network does not sufficiently approximate the generating distribution. And although the posterior distribution may concentrate to the best approximation of the true distribution in KL divergence in such misspecified models, there exists pathological cases where the regular Bayesian posterior is not consistent at all, see \cite{grunwaldmisspecifiation}. This is the reason why we focus here on tempered posteriors which are robust to such misspecification. Therefore, we introduce in Section \ref{section-Framework} a notion of consistency of a Bayesian estimator which is closely related to the notion of concentration - even stronger - and which enables a more robust formulation of generalization error bounds for variational approximations. See Appendix \ref{app:connection} for more details on the connection between the notions of consistency and concentration. 

Then we focus on optimization aspects. We no longer assume an ideal optimization, as done for instance in \cite{SchmidtHieberDNN,Imaizumi19DNN}. We address in this paper the question of the consistency of numerical algorithms used to compute our ideal approximations. We consider an optimization error given by any algorithm and independent to the statistical error, and we show how it affects our generalization result. Our upper bound highlights the connection between the consistency of the variational approximation and the convergence of the ELBO.

We also provide insights on the structure of the network which leads to optimal rates of convergence, i.e. its depth, its width and its sparsity. Indeed, in our first generalization error bound, the structure of the network is ideally tuned for some choice of the generating function, and we show how to choose such a structure. Nevertheless, the characteristics of the regression function may be unknown, e.g. we may know that the regression function is H\"older continuous but we ignore its level of smoothness. We propose here an automated method for choosing the architecture of the network. We introduce a classical model selection framework based on the ELBO criterion \citep{cherief2018consistency2}, and we show that the variational approximation associated with the selected structure does not overfit and adaptively achieves the optimal rate of convergence even without any oracle information.

The rest of this paper is organized as follows. Section \ref{section-Framework} introduces the notations and the framework that will be considered in the paper, and presents sparse spike-and-slab variational inference for deep neural networks. Section \ref{section-DeepVI} provides theoretical generalization error bounds for variational approximations of DNNs and shows the optimality of the method for estimating H\"older smooth functions. Finally, insights on the choice of the architecture of the network are given in Section \ref{section-Model-selection} via the ELBO maximization framework. All the proofs are deferred to the appendix.

\section{Sparse deep variational inference}
\label{section-Framework}
Let us introduce the notations and the statistical framework we adopt in this paper. For any vector $x=(x_1,...,x_d) \in [-1,1]^d$ and any real-valued function $f$ defined on $[-1,1]^d$, $d>0$, we denote:
$$
\|x\|_\infty = \max_{1\leq i \leq d} |x_i|
\hspace{0.5cm}
\textnormal{,}
\hspace{0.5cm}
\|f\|_2 = \bigg( \int f^2 \bigg)^{1/2}
\hspace{0.5cm}
\textnormal{and}
\hspace{0.5cm}
\|f\|_\infty = \sup_{y\in [-1,1]^d} |f(y)|.
$$
For any $\mathbf{k}\in \{0,1,2,...\}^d$, we define $|\mathbf{k}|=\sum_{i=1}^d k_i$ and the mixed partial derivatives when all partial derivatives up to order $|\mathbf{k}|$ exist:
$$
D^{\mathbf{k}} f(x) = \frac{\partial^{|\mathbf{k}|}f}{\partial^{k_1}x_1...\partial^{k_d}x_d} (x) .
$$

We also introduce the notion of $\beta$-H\"older continuity for $\beta>0$. We denote $\lfloor \beta \rfloor$ the largest integer strictly smaller than $\beta$. Then $f$ is said to be $\beta$-H\"older continuous \citep{tsybakov2008} if all partial derivatives up to order $\lfloor \beta \rfloor$ exist and are bounded, and if:
$$
\|f\|_{\mathcal{C}_\beta} := \max_{|\mathbf{k}|\leq \lfloor \beta \rfloor} \|D^{\mathbf{k}} f\|_\infty + \max_{|\mathbf{k}| = \lfloor \beta \rfloor} \sup_{x,y \in [-1,1]^d, x\ne y} \frac{|D^\mathbf{k} f(x)-D^\mathbf{k} f(y)|}{\|x-y\|_\infty^{\beta-\lfloor \beta \rfloor}} < +\infty.
$$
$\|f\|_{\mathcal{C}_\beta}$ is the norm of the H\"older space $\mathcal{C}_\beta = \{f/\|f\|_{\mathcal{C}_\beta}<+\infty\}$.

\subsection{Nonparametric regression}

We consider the nonparametric regression framework. We have a collection of random variables $(X_i,Y_i) \in [-1,1]^d\times \mathbb{R}$ for $i=1,...,n$ which are independent and identically distributed (i.i.d.) with the generating process:
$$
\begin{cases}
X_i \sim \mathcal{U}([-1,1]^d), \\
Y_i = f_0(X_i)+\zeta_i
\end{cases}
$$
where $\mathcal{U}([-1,1]^d)$ is the uniform distribution on the interval $[-1,1]^d$, $\zeta_1,...,\zeta_n$ are i.i.d. Gaussian random variables with mean $0$ and known variance $\sigma^2$, and $f_0:[-1,1]^d\rightarrow \mathbb{R}$ is the true unknown function. For instance, the true regression function $f_0$ may belong to the set $\mathcal{C}_\beta$ of H\"older functions with level of smoothness $\beta$.

\subsection{Deep neural networks}

We call deep neural network any map $f_\theta:\mathbb{R}^d\rightarrow \mathbb{R}$ defined recursively as follows:
$$
\begin{cases}
x^{(0)}:=x, \\
x^{(\ell)}:= \rho(A_\ell x^{(\ell-1)} + b_\ell) \hspace{0.3cm} \text{for} \hspace{0.1cm} \ell=1,...,L-1, \\
f_\theta(x):=A_L x^{(L-1)} + b_L
\end{cases}
$$
where $L\geq 3$. $\rho$ is an activation function acting componentwise. For instance, we can choose the ReLU activation function $\rho(u)=\max(u,0)$. Each $A_\ell \in \mathbb{R}^{D_\ell\times D_{\ell-1}}$ is a weight matrix such that its $(i,j)$ coefficient, called edge weight, connects the $j$-th neuron of the $(\ell-1)$-th layer to the $i$-th neuron of the $\ell$-th layer, and each $b_\ell \in \mathbb{R}^{D_\ell}$ is a shift vector such that its $i$-th coefficient, called node vector, represents the weight associated with the $i$-th node of layer $\ell$. We set $D_0=d$ the number of units in the input layer, $D_{L}=1$ the number of units in the output layer and $D_\ell=D$ the number of units in the hidden layers. The architecture of the network is characterized by its number of edges $S$, i.e. the total number of nonzero entries in matrices $A_\ell$ and vectors $b_\ell$, its number of layers $L\geq 3$ (excluding the input layer), and its width $D \geq 1$. We have $S\leq T$ where $T=\sum_{\ell=1}^L D_\ell(D_{\ell-1}+1)$ is the total number of coefficients in a fully connected network. By now, we consider that $S$, $L$ and $D$ are fixed, and $d=\mathcal{O}(1)$ as $n\rightarrow +\infty$. In particular, we assume that $d \leq D$, which implies that $T \leq LD(D+1)$. We also suppose that the absolute values of all coefficients are upper bounded by some positive constant $B\geq 2$. This boundedness assumption will be relaxed in the appendix, see Appendix \ref{app:Gauss}. Then, the parameter of a DNN is $\theta = \{ (A_1,b_1),...,(A_L,b_L) \}$, and we denote $\Theta_{S,L,D}$ the set of all possible parameters. We will also alternatively consider the stacked coefficients parameter $\theta=(\theta_1,...,\theta_T)$.

\subsection{Bayesian modeling}

We adopt a Bayesian approach, and we place a spike-and-slab prior $\pi$ \citep{Castillo2015SS} over the parameter space $\Theta_{S,L,D}$ (equipped with some suited sigma-algebra) that is defined hierarchically. The spike-and-slab prior is known to be a relevant alternative to dropout for Bayesian deep learning, see \cite{Rockova2018}. First, we sample a vector of binary indicators $\gamma=(\gamma_1,...,\gamma_T)\in \{0,1\}^T$ uniformly among the set $\mathcal{S}^S_T$ of $T$-dimensional binary vectors with exactly $S$ nonzero entries, and then given $\gamma_t$ for each $t=1,...,T$, we put a spike-and-slab prior on $\theta_t$ that returns $0$ if $\gamma_t=0$ and a random sample from a uniform distribution on $[-B,B]$ otherwise:
$$
\begin{cases}
\gamma \sim \mathcal{U}(\mathcal{S}^S_T), \\
\theta_t|\gamma_t \sim \gamma_t \hspace{0.1cm} \mathcal{U}([-B,B]) + (1-\gamma_t)\delta_{\{0\}}, \hspace{0.2cm} t=1,...,T
\end{cases}
$$
where $\delta_{\{0\}}$ is a point mass at $0$ and $\mathcal{U}([-B,B])$ is a uniform distribution on $[-B,B]$. We recall that the sparsity level $S$ is fixed here and that this assumption will be relaxed in Section \ref{section-Model-selection}.

\begin{rmk}
We consider uniform distributions for simplicity as in similar works \citep{Rockova2018,Suzuki18DNNkerels}, but Gaussian distributions can be used as well when working on an unbounded parameter set $\Theta_{S,L,D}$, see Theorem \ref{thm-Gauss} in Appendix \ref{app:Gauss}.
\end{rmk} 

Then we define the tempered posterior distribution $\pi_{n,\alpha}$ on parameter $\theta \in \Theta_{S,L,D}$ using prior $\pi$ for any $\alpha \in (0,1)$:
\[
\pi_{n,\alpha}(d\theta) \propto \exp\bigg(-\frac{\alpha}{2\sigma^2}\sum_{i=1}^n (Y_i-f_\theta(X_i))^2\bigg) \pi(d\theta),
\]
which is a slight variant of the definition of the regular Bayesian posterior (for which $\alpha=1$). This distribution is known to be easier to sample from, to require less stringent assumptions to obtain concentration, and to be robust to misspecification, see respectively \cite{behrens2012tuning}, \cite{bhattacharya2016bayesian} and \cite{grunwaldmisspecifiation}.

\subsection{Sparse variational inference}

The variational Bayes approximation $\tilde{\pi}_{n,\alpha}$ of the tempered posterior is defined as the projection (with respect to the Kullback-Leibler divergence) of the tempered posterior onto some set $\mathcal{F}_{S,L,D}$:
$$
\tilde{\pi}_{n,\alpha} = \argmin_{q \in \mathcal{F}_{S,L,D}}  \textnormal{KL}(q\|\pi_{n,\alpha}).
$$
which is equivalent to:
\begin{equation}
\label{VBdefinition}
\tilde{\pi}_{n,\alpha} = \argmin_{q \in \mathcal{F}_{S,L,D}} \bigg\{ \frac{\alpha}{2\sigma^2} \sum_{i=1}^n \int (Y_i-f_\theta(X_i))^2 q(d\theta) + \textrm{KL}(q\|\pi) \bigg\}
\end{equation}
where the function inside the argmin operator in \eqref{VBdefinition} is the opposite of the evidence lower bound $\mathcal{L}_n(q)$.

We choose a sparse spike-and-slab variational set $\mathcal{F}_{S,L,D}$ - see for instance \cite{VariationalSparseCoding2019} - which can be seen as an extension of the popular mean-field variational set with a dependence assumption specifying the number of active neurons. The mean-field approximation is based on a decomposition of the space of parameters $\Theta_{S,L,D}$ as a product $\theta=(\theta_1,...,\theta_T)$ and consists in compatible product distributions on each parameter $\theta_t$, $t=1,...,T$. Here, we fit a distribution in the family that matches the prior: we first choose a distribution $\pi_\gamma$ on the set $\mathcal{S}^S_T$ that selects a $T$-dimensional binary vector $\gamma$ with $S$ nonzero entries, and then we place a spike-and-slab variational approximation on each $\theta_t$ given $\gamma_t$:
$$
\begin{cases}
\gamma \sim \pi_{\gamma}, \\
\theta_t|\gamma_t \sim \gamma_t \hspace{0.1cm} \mathcal{U}([l_t,u_t]) + (1-\gamma_t)\delta_{\{0\}} \hspace{0.3cm} \textnormal{for each} \hspace{0.2cm} t=1,...,T
\end{cases}
$$
where $-1 \leq l_t \leq u_t \leq 1$, with the distribution $\pi_{\gamma}$ and the intervals $[l_t,u_t]$, $t=1,...,T$ as the hyperparameters of the variational set $\mathcal{F}_{S,L,D}$. In particular, if we choose a deterministic $\pi_{\gamma}=\delta_{\{\gamma'\}}$ with $\gamma' \in \mathcal{S}^S_T$, then we will obtain a parametric mean-field approximation. See Section 6.6 of the PhD thesis of \cite{YarinGalThesis} for a more detailed discussion on the connection between Gaussian mean-field and sparse spike-and-slab posterior approximations.

The generalization error of the tempered posterior ${\pi}_{n,\alpha}$ and of its variational approximation $\tilde{\pi}_{n,\alpha}$ is the expected average of the squared $L_2$-distance to the true generating function over the Bayesian estimator:
$$
\mathbb{E} \bigg[ \int \|f_\theta-f_0\|_2^2 {\pi}_{n,\alpha}(d\theta) \bigg] 
\hspace{0.5cm}
\textnormal{and}
\hspace{0.5cm}
\mathbb{E} \bigg[ \int \|f_\theta-f_0\|_2^2 \tilde{\pi}_{n,\alpha}(d\theta) \bigg].
$$
We say that a Bayesian estimator is consistent at rate $r_n\rightarrow 0$ if its generalization error is upper bounded by $r_n$. Notice that consistency of the Bayesian estimator implies concentration to $f_0$. Again, see Appendix \ref{app:connection} for the connection between these two notions.


\section{Generalization of variational inference for neural networks}
\label{section-DeepVI}
The first result of this section is an extension of the result of \cite{Rockova2018} on the Bayesian distribution for H\"older regression functions. Indeed, we provide a concentration result on the posterior distribution for the expected $L_2$-distance instead of the empirical $L_2$-distance, which enables generalization instead of reconstruction on the training datapoints. This result is then extended again to the variational approximation for our definition of consistency: we show that we can still achieve near-optimality using an approximation of the posterior without any additional assumption. Finally, we explain how we can incorporate optimization error in our generalization results.

\subsection{Concentration of the posterior}

\cite{Rockova2018} gives the first posterior concentration result for deep ReLU networks when estimating H\"older smooth functions in nonparametric regression with empirical $L_2$-distance. The authors highlight the flexibility of DNNs over other methods for estimating $\beta$-H\"older smooth functions as there is a large range of values of the level of smoothness $\beta$ for which one can obtain concentration, e.g. $0<\beta<d$ for a DNN against $0<\beta<1$ for a Bayesian tree.

The following theorem provides the concentration of the tempered posterior distribution $\pi_{n,\alpha}$ for deep ReLU neural networks when using the expected $L_2$-distance for some suitable architecture of the network:

\begin{thm}
 \label{thm-concentration}
Let us assume that $\alpha \in (0,1)$, that $f_0$ is $\beta$-H\"older smooth with $0<\beta<d$ and that the activation function is ReLU. We consider the architecture of \cite{Rockova2018} for some positive constant $C_D$ independent of $n$:
$$
L= 8 + (\lfloor\log_2n\rfloor + 5)(1 + \lceil\log_2 d\rceil),
$$
$$
D = C_D \lfloor n^{\frac{d}{2\beta+d}}/\log n \rfloor,
$$
$$
S \leq 94 d^2 (\beta+1)^{2d} D (L + \lceil\log_2 d\rceil).
$$
Then the tempered posterior distribution $\pi_{n,\alpha}$ concentrates at the
minimax rate $r_n = n^{\frac{-2\beta}{2\beta+d}}$ up to a (squared) logarithmic factor for the expected $L_2$-distance in the sense that:
\begin{equation*}
\label{thm-concentration-holder}
{\pi}_{n,\alpha}\bigg(\theta\in\Theta_{S,L,D} \hspace{0.1cm} \big/ \hspace{0.1cm} \|f_\theta-f_0\|_2^2>M_n \cdot n^{\frac{-2\beta}{2\beta+d}} \cdot \log^2 n \bigg) \xrightarrow[n\rightarrow +\infty]{} 0
\end{equation*}
in probability as $n\rightarrow+\infty$ for any $M_n\rightarrow +\infty$.
\end{thm}

In order to prove Theorem \ref{thm-concentration}, we actually have to check that the so-called \textit{prior mass} condition is satisfied:
\begin{equation}
\label{prior-mass-condition}
    {\pi}\bigg(\theta\in\Theta_{S,L,D} \hspace{0.1cm} \big/ \hspace{0.1cm} \|f_\theta-f_0\|_2^2 \leq r_n \bigg) \geq e^{-nr_n}.
\end{equation}
This assumption, introduced in \cite{ghosal2000convergence} in order to obtain the concentration of the regular posterior distribution states that the prior must give enough mass to some neighborhood of the true parameter. As shown in \cite{bhattacharya2016bayesian}, this condition is even sufficient for tempered posteriors. Actually, this inequality was first stated using the KL divergence instead of the expected $L_2$-distance (see Condition 2.4 in Theorem 2.1 in \cite{ghosal2000convergence}), but the KL metric is equivalent to the squared $L_2$-metric in regression problems
with Gaussian noise. This prior mass condition gives us the rate of convergence of the tempered posterior $r_n=n^{\frac{-2\beta}{2\beta+d}}$ (up to a squared logarithmic factor) which is known to be optimal when estimating $\beta$-H\"older smooth functions \citep{tsybakov2008}. Note that the $\log^2 n$ term is common in the theoretical deep learning literature \citep{Imaizumi19DNN,suzuki2019adaptivity,SchmidtHieberDNN}.

\begin{rmk}
The number of parameters of order $n^{\frac{2d}{2\beta+d}}/\log n \in [n^{2/3}/\log(n),n^2/\log(n)]$ is high compared to standard machine learning methods, which may lead to overfitting and hence prevent the procedure from achieving the minimax rate of convergence. The sparsity parameter $S$ which gives a network with a small number of nonzero parameters along with the spike-and-slab prior help us tackle this issue and obtain optimal rates of convergence (up to logarithmic factors). 
\end{rmk}

\subsection{A generalization error bound}

The result we state in this subsection applies to a wide range of activation functions, including the popular ReLU activation and the identity map:

\begin{asm}
\label{asm1}
In the following, we assume that the activation function $\rho$ is $1$-Lispchitz continuous (with respect to the aboluste value) and is such that for any $x\in\mathbb{R}$, $|\rho(x)|\leq|x|$.
\end{asm}

We do not assume any longer that the regression function is $\beta$-H\"older and we consider any structure $(S,L,D)$. The following theorem gives a generalization error bound when using variational approximations instead of exact tempered posteriors for DNNs. The proof is given in Appendix \ref{app:proof-thm-2} and is based on PAC-Bayes theory \citep{CatoniThermo,guedj2019primer}:

\begin{thm}
 \label{thm-general}
For any $\alpha\in(0,1)$,
\begin{equation}
\label{main-thm-fixed}
 \mathbb{E} \bigg[ \int \|f_\theta-f_0\|_2^2 \tilde{\pi}_{n,\alpha}(d\theta) \bigg] \leq \frac{2}{1-\alpha} \inf_{\theta^* \in \Theta_{S,L,D}} \|f_{\theta^*}-f_0\|_2^2 + \frac{2}{1-\alpha} \bigg(1+\frac{\sigma^2}{\alpha}\bigg) r_n^{S,L,D} ,  
\end{equation}
with 
$$
r_n^{S,L,D} = \frac{LS}{n}\log(BD) + \frac{2S}{n}\log(BLD) + \frac{S}{n} \log\bigg(7dL\max\big(\frac{n}{S},1\big)\bigg).
$$
\end{thm}

The oracle inequality \eqref{main-thm-fixed} ensures consistency of variational Bayes for estimating neural networks and provides the associated rate of convergence given the structure $(S,L,D)$. Indeed, if $f_0$ is a neural network with structure $(S,L,D)$, then the infimum term on the right hand side of the inequality vanishes and we obtain a rate of convergence of order 
$$
r_n^{S,L,D} \sim \max\bigg(\frac{S\log (nL/S)}{n},\frac{LS\log D}{n}\bigg),
$$
which underlines a linear dependence on the number of layers and the sparsity. In fact, this rate of convergence is determined by the \textit{extended prior mass condition} \citep{Tempered,cherief2018consistency,cherief2018consistency2}, which requires that in addition to the previous prior mass condition of \cite{ghosal2000convergence} and \cite{bhattacharya2016bayesian}, the variational set $\mathcal{F}_{S,L,D}$ must contain probability distributions $q$ that are concentrated enough around the true generating function $f_0$. One of the main findings of Theorem \ref{thm-general} is that our choice of the sparse spike-and-slab variational set $\mathcal{F}_{S,L,D}$ is rich enough and that both conditions are actually similar and lead to the same rate of convergence. Hence, the rate of convergence is the one that satisfies the prior mass condition \eqref{prior-mass-condition}. In particular, as the prior distribution is uniform over the parameter space, the negative logarithm of the prior mass of the neighborhood of the true regression function in Equation \eqref{prior-mass-condition} is a local covering entropy, that is the logarithm of the number of $r_n^{S,L,D}$-balls needed to cover a neighborhood of the true regression function. Especially, it has been shown in previous studies that this local covering entropy fully characterizes the rate of convergence of the empirical risk minimizer for DNNs \citep{SchmidtHieberDNN,suzuki2019adaptivity}. The rate $r_n^{S,L,D}$ we obtain in this work is exactly of the same order than the upper bound on the covering entropy number given in Lemma 5 in \cite{SchmidtHieberDNN} and in Lemma 3 in \cite{suzuki2019adaptivity} which derive rates of convergence for the empirical risk minimizer using different proof techniques. Note that replacing a uniform by a Gaussian in the prior and variational distributions leads to the same rate of convergence, see Appendix \ref{app:Gauss}.

Nevertheless, deep neural networks are mainly used for their computational efficiency and their ability to approach complex functions, which makes the task of estimating a neural network not so popular in machine learning. As said earlier, \cite{Imaizumi19DNN} used neural networks for estimating non-smooth functions. In such a context where the neural network model is misspecified, our generalization error bound is robust and still holds, and satisfies the best possible balance between bias and variance.

Indeed, the upper bound on the generalization error on the right-hand-side of \eqref{main-thm-fixed} is mainly divided in two parts: the approximation error of $f_0$ by a DNN $f_{\theta^*}$ in $\Theta_{S,L,D}$ (i.e. the bias) and the estimation error $r_n^{S,L,D}$ of a neural network $f_{\theta^*}$ in $\Theta_{S,L,D}$ (i.e. the variance). For instance, even if the generalization power is decreasing linearly with respect to the number of layers compared to the logarithmic dependence on the width due to the variance term, this effect is compensated by the benefits of depth in the approximation theory of deep learning. Then, as there exists relationships between the bias/the variance and the architecture of a neural network (respectively due to the approximation theory/the form of $r_n^{S,L,D}$), Theorem \ref{thm-general} gives both a general formula for deriving rates of convergence for variational approximations and insight on the way to choose the architecture. We choose the architecture that minimizes the right-hand-side of \eqref{main-thm-fixed}, which can lead to minimax estimators for smooth functions. It also connects the approximation and estimation theories following previous studies. This was done for instance by \cite{SchmidtHieberDNN,suzuki2019adaptivity,Imaizumi19DNN} who exploited the effectiveness of ReLU activation function in terms of approximation ability \citep{Yarotsky2017,Petersen2018Approximation} for H\"older/Besov smooth and piecewise smooth generating functions. 

Now we illustrate Theorem \ref{thm-general} on H\"older smooth functions. The following result shows that the variational approximation achieves the same rate of convergence than the posterior distribution it approximates, and even the minimax rate of convergence if the architecture is well chosen. We present both consistency and concentration results.

\begin{cor}
 \label{cor-concentration}
 Let us fix $\alpha \in (0,1)$. We consider the ReLU activation function.
Assume that $f_0$ is $\beta$-H\"older smooth with $0<\beta<d$. Then with $L$, $D$ and $S$ defined as in Theorem \ref{thm-concentration}, the variational approximation of the tempered posterior distribution $\tilde{\pi}_{n,\alpha}$ is consistent and hence concentrates at the minimax rate $r_n = n^{\frac{-2\beta}{2\beta+d}}$ (up to a squared logarithmic factor):
\begin{equation*}
\label{thm-concentration-holder}
\tilde{\pi}_{n,\alpha}\bigg(\theta\in\Theta_{S,L,D} \hspace{0.1cm} \big/ \hspace{0.1cm} \|f_\theta-f_0\|_2^2>M_n \cdot n^{\frac{-2\beta}{2\beta+d}} \cdot \log^2 n \bigg) \xrightarrow[n\rightarrow +\infty]{} 0
\end{equation*}
in probability as $n\rightarrow+\infty$ for any $M_n\rightarrow +\infty$.

\end{cor}

\subsection{Optimization error}

In this subsection, we discuss the effect of an optimization error that is independent on the previous statistical error. Indeed, in the variational Bayes community, people use approximate algorithms in practice to solve the optimization problem \eqref{VBdefinition} when the model is non-conjugate, i.e. the VB solution is not available in closed-form. This is the case here when considering a sparse spike-and-slab variational approximation in $\mathcal{F}_{S,L,D}$ for DNNs with hyperparameters $\phi=(\pi_{\gamma},(\phi_t)_{1\leq t\leq T})$ and an algorithm that gives a sequence of hyperparameters $(\phi^k)_{k\geq 1}$ and associated variational approximations $(\tilde{\pi}_{n,\alpha}^k)_{k\geq 1}$. The following theorem gives a statistical guarantee for any approximation $\tilde{\pi}_{n,\alpha}^k$, $k\geq 1$:

\begin{thm}
 \label{thm-opti}
For any $\alpha\in(0,1)$,
\begin{align*}
 \mathbb{E} \bigg[ \int \|f_\theta-f_0\|_2^2 \tilde{\pi}_{n,\alpha}^k(d\theta) \bigg] \leq \frac{2}{1-\alpha} \inf_{\theta^*}  \|f_{\theta^*}-f_0\|_2^2 + \frac{2}{1-\alpha} &\bigg(1+\frac{\sigma^2}{\alpha}\bigg) r_n^{S,L,D} \\ 
 & + \frac{2\sigma^2}{\alpha(1-\alpha)} \cdot \frac{\mathbb{E}[\mathcal{L}^*_n-\mathcal{L}_n^k]}{n},
\end{align*}
where $\mathcal{L}^*_n$ is the maximum of the evidence lower bound i.e. the ELBO evaluated at $\tilde{\pi}_{n,\alpha}$, while $\mathcal{L}_n^k$ is the ELBO evaluated at $\tilde{\pi}_{n,\alpha}^k$.
\end{thm}

We establish a clear connection between the convergence (in mean) of the ELBO $\mathcal{L}_n^k$ to $\mathcal{L}^*_n$ and the consistency of our algorithm $\tilde{\pi}_{n,\alpha}^k$. Indeed, as soon as the ELBO $\mathcal{L}_n^k$ converges at rate $c_{k,n}$, then our variational approximation $\tilde{\pi}_{n,\alpha}^k$ is consistent at rate:
$$
\max\bigg(\frac{c_{k,n}}{n},\frac{S\log (nL/S)}{n},\frac{SL\log D}{n}\bigg).
$$
In particular, as soon as $k$ is such that $c_{k,n}\leq\max(S\log n,S\log D)$, then we obtain consistency of $\tilde{\pi}_{n,\alpha}^k$ at rate $r_n^{S,L,D}$, i.e. $\tilde{\pi}_{n,\alpha}^k$ and $\tilde{\pi}_{n,\alpha}$ have the same rate of convergence.

However, deriving the convergence of the ELBO is a hard task. For instance, when considering a simple Gaussian mean-field approximation without sparsity, the variational objective $\mathcal{L}_n$ can be maximized using either stochastic \citep{GravesVI2011,Blundell2015} or natural gradient methods \citep{KhanBayesianDeepLearning2018} on the parameters of the Gaussian approximation. The convergence of the ELBO is often met in practice \citep{QuasiMonteCarloBuchholz18a,SlangKhan2018} and the recent work of \cite{KhanBayesianDeepLearning2019} even showed that Bayesian deep learning enables practical deep learning and matches the performance of standard methods while preserving benefits of Bayesian principles. Nevertheless, the objective is nonconvex and hence it is difficult to prove the convergence to a global maximum in theory. Some recent papers studied global convergence properties of gradient descent algorithms for frequentist classification and regression losses \citep{GlobalMinimaDNNDu,ConvergenceDeepLearningAllenZhu} that we may extend to gradient descent algorithms for the ELBO objective such as Variational Online Gauss Newton or Vadam \citep{KhanBayesianDeepLearning2018,KhanBayesianDeepLearning2019}. 

Another point is to develop and study more complex algorithms than simple gradient descent that deal with spike-and-slab sparsity-inducing variational inference, as for instance \cite{SpikeAndSlabVITitsias2011} did for multi-task and multiple kernel learning. Also, \cite{Louizos2018SparseDL} connected sparse spike-and-slab variational inference with $L_0$-norm regularization for neural networks and proposed a solution to the intractability of the $L_0$-penalty term through the use of non-negative stochastic gates, while \cite{DeepRewiring2018} proposed an algorithm preserving sparsity during training. Nevertheless, these optimization concerns fall beyond the scope of this paper and are left for further research.

\section{Architecture design via ELBO maximization}
\label{section-Model-selection}
We saw in Section \ref{section-DeepVI} that the choice of the architecture of the neural network is crucial and can lead to faster convergence and better approximation. In this section, we formulate the architecture design of DNNs as a model selection problem and we investigate the ELBO maximization strategy which is very popular in the variational Bayes community. This approach is different from \cite{Rockova2018} which is fully Bayesian and treats the parameters of the network architecture, namely the depth, the width and the sparsity, as random variables. We show that the ELBO criterion does not overfit and is adaptive: it provides a variational approximation with the optimal rate of convergence, and it does not require the knowledge of the unknown aspects of the regression function $f_0$ (e.g. the level of smoothness for smooth functions) to select the optimal variational approximation.

We denote $\mathcal{M}_{S,L,D}$ the statistical model associated with the parameter set $\Theta_{S,L,D}$. We consider a countable number of models, and we introduce prior beliefs $\pi_{S,L,D}$ over the sparsity, the depth and the width of the network, that can be defined hierarchically and that are known beforehand. For instance, the prior beliefs can be chosen such that $\pi_L=2^{-L}$, $\pi_{D|L}$ follows a uniform distribution over $\{d,...,\max(e^L,d)\}$ given $L$, and $\pi_{S|L,D}$ a uniform distribution over $\{1,...,T\}$ given $L$ and $D$ (we recall that $T$ is the number of coefficients in a fully connected network). This particular choice is sensible as it allows to consider any number of hidden layers and (at most) an exponentially large width with respect to the depth of the network. We still consider spike-and-slab priors on $\theta_{S,L,D} \in \Theta_{S,L,D}$ given model $\mathcal{M}_{S,L,D}$.

Each tempered posterior associated with model $\mathcal{M}_{S,L,D}$ is denoted $\pi_{n,\alpha}^{S,L,D}$. We recall that the variational approximation $\tilde{\pi}_{n,\alpha}^{S,L,D}$ associated with model $\mathcal{M}_{S,L,D}$ is defined as the distribution into the variational set $\mathcal{F}_{S,L,D}$ that maximizes the Evidence Lower Bound:
\begin{equation*}
\tilde{\pi}_{n,\alpha}^{S,L,D} = \argmax_{q^{S,L,D} \in \mathcal{F}_{S,L,D}} \mathcal{L}_n(q^{S,L,D}).
\end{equation*}
We will simply denote in the following $\mathcal{L}^*_n(S,L,D)$ the closest approximation to the log-evidence i.e., the value of the ELBO evaluated at its maximum: 
$$
\mathcal{L}_n^*(S,L,D)=\mathcal{L}_n(\tilde{\pi}_{n,\alpha}^{S,L,D}).
$$

The model selection criterion we use here to select the architecture of the network is a slight penalized variant of the classical ELBO criterion \citep{Blei2017ReviewVB} with strong theoretical guarantees \citep{cherief2018consistency2} :
\begin{equation*}
\label{ELBOcriterion}
(\hat{S},\hat{L},\hat{D}) = \argmax_{S,L,D} \bigg\{ \mathcal{L}^*_n(S,L,D) - {\log\bigg(\frac{1}{\pi_{S,L,D}}\bigg)} \bigg\}.
\end{equation*}
For any choice of the prior beliefs $\pi_{S,L,D}$, compute the ELBO for each model $\mathcal{M}_{S,L,D}$ using an algorithm that will converge to $\mathcal{L}^*_n(S,L,D)$ and choose the architecture that maximizes the penalized ELBO criterion. It is possible to restrict to a finite number of layers in practice (for instance, a factor of $n$ or $\log n$).

The following theorem shows that this ELBO criterion leads to a variational approximation with the optimal rate of convergence:

\begin{thm}
 \label{thm-model-selection}
For any $\alpha\in(0,1)$,
\begin{align*}
 \mathbb{E} \bigg[ \int \|f_\theta-f_0\|_2^2 \tilde{\pi}_{n,\alpha}^{\hat{S},\hat{L},\hat{D}}(d\theta) \bigg] \leq \inf_{S,L,D} \bigg\{ \frac{2}{1-\alpha} \inf_{\theta^* \in \Theta_{S,L,D}} \|f_{\theta^*} -f_0\|_2^2 & + \frac{2}{1-\alpha} \bigg(1+\frac{\sigma^2}{\alpha}\bigg) r_n^{S,L,D} \\
 & + \frac{2\sigma^2}{\alpha(1-\alpha)} \frac{\log(\frac{1}{\pi_{S,L,D}})}{n} \bigg\} .
\end{align*}
\end{thm}

This inequality shows that as soon as the complexity term $\log(1/\pi_{S,L,D})/n$ that reflects the prior beliefs is lower than the effective rate of convergence that balances the accuracy and the estimation error $r_n^{S,L,D}$, the selected variational approximation adaptively achieves the best possible rate. For instance, it leads to (near-)minimax rates for H\"older smooth functions and selects the optimal architecture even without the knowledge of $\beta$, which was required in the previous section. Note that for the previous choice of prior beliefs $\pi_L=2^{-L}$, $\pi_{D|L}=1/(\max(e^L,d)-d+1)$, $\pi_{S|L,D}=1/T$, we get:
$$
\frac{\log(\frac{1}{\pi_{S,L,D}})}{n} \leq \frac{2\log (D+1) + \log L + \max(L,\log d) + L \log2}{n}
$$
that is lower than $r_n^{S,L,D}$ (up to a factor) and hence the ELBO criterion does not overfit.

\section{Discussion}
\label{section-Discussion}
In this paper, we provided theoretical justifications for neural networks from a Bayesian point of view using sparse variational inference. We derived new generalization error bounds and we showed that sparse variational approximations of DNNs achieve (near-)minimax optimality when the regression function is H\"older smooth. All our results directly imply concentration of the approximation of the posterior distribution. We also proposed an automated method for selecting an architecture of the network with optimal consistency guarantees via the ELBO maximization framework.

We think that one of the main challenges here is the design of new computational algorithms for spike-and-slab deep learning in the wake of the work of \cite{SpikeAndSlabVITitsias2011} for multi-task and multiple kernel learning, or those of \cite{Louizos2018SparseDL} and \cite{DeepRewiring2018}. In the latter paper, the authors designed an algorithm for training deep networks while simultaneously learning their sparse connectivity allowing for fast and computationally efficient learning, whereas most approaches have focused on compressing already trained neural networks.

In the same time, a future point of interest is the study of the global convergence of these approximate algorithms in nonconvex settings i.e. study of the theoretical convergence of the ELBO. This work was conducted for frequentist gradient descent algorithms \citep{ConvergenceDeepLearningAllenZhu,GlobalMinimaDNNDu}. Such studies should be investigated for Bayesian gradient descents, as well as for algorithms that preserve the sparsity of the network during training.

\vspace{-0.2cm}

\acks{We would like to warmly thank Pierre Alquier for his helpful suggestions on early versions of this work. }

\vspace{-0.2cm}


\newpage

\vskip 0.2in

\newpage

\newpage

\appendix
\section{Connection between concentration and consistency}
\label{app:connection}
In this appendix, we show the connection between the notions of \textit{consistency} and \textit{concentration}. 

The Bayesian estimator $\rho$ (e.g. the tempered posterior $\pi_{n,\alpha}$ or its variational approximation $\tilde{\pi}_{n,\alpha}$) is said to be consistent if its generalization error goes to zero as $n\rightarrow +\infty$:
$$
\mathbb{E} \bigg[ \int \| f_{\theta} - f_0 \|_2^2 \rho(d\theta) \bigg] \xrightarrow[n\rightarrow +\infty]{} 0.
$$
We say that the Bayesian estimator $\rho$ concentrates at rate $r_n$ \citep{ghosal2000convergence} if in probability (with respect to the random variables distributed according to the generating process), the estimator concentrates asymptotically around the true distribution as $n\rightarrow +\infty$, i.e.:
\begin{equation*}
\rho\bigg(\theta\in\Theta_{S,L,D} \hspace{0.1cm} \big/ \hspace{0.1cm} \|f_\theta-f_0\|_2^2>M_n r_n\bigg) \xrightarrow[n\rightarrow +\infty]{} 0.
\end{equation*}
in probability as $n\rightarrow+\infty$ for any $M_n\rightarrow +\infty$.

The consistency of the Bayesian distribution $\rho$ at rate $r_n$ implies its concentration at rate $r_n$. Indeed, if we we assume that $\rho$ is consistent at rate $r_n$, i.e.:
$$
\mathbb{E} \bigg[ \int \|f_\theta-f_0\|_2^2 \rho(d\theta) \bigg] \leq r_n,
$$
then, using Markov's inequality for any $M_n \rightarrow +\infty$ as $n\rightarrow+\infty$:
$$
\mathbb{E} \bigg[ \rho\bigg(\theta\in\Theta_{S,L,D} \hspace{0.1cm} \big/ \hspace{0.1cm} \|f_\theta-f_0\|_2^2>M_n r_n\bigg) \bigg] \leq \frac{\mathbb{E} \bigg[ \int \|f_\theta-f_0\|_2^2 \rho(d\theta) \bigg]}{M_n r_n} \leq \frac{r_n}{M_n r_n} = \frac{1}{M_n} \rightarrow 0.
$$
Hence, we have the convergence in mean of $\rho\big(\theta\in\Theta_{S,L,D} \hspace{0.1cm} \big/ \hspace{0.1cm} \|f_\theta-f_0\|_2^2>M_n r_n\big)$ to $0$, and then the convergence in probability of $\rho\big(\theta\in\Theta_{S,L,D} \hspace{0.1cm} \big/ \hspace{0.1cm} \|f_\theta-f_0\|_2^2>M_n r_n\big)$ to $0$, i.e.\ the concentration of $\rho$ to $f_0$ at rate $r_n$.

\section{Proof of Theorem \ref{thm-general}}
\label{app:proof-thm-2}
The structure of the proof of Theorem \ref{thm-general} is composed of three main steps. The first one consists in obtaining the general shape of the inequality using PAC-Bayes inequalities, and the two others in finding a rate that satisfies the extended prior mass condition.

\vspace{0.5cm}

\underline{\textit{First step}}: \underline{\textit{we obtain the general inequality}}

\vspace{0.5cm}

We start from inequality 2.6 in \cite{Tempered} that provides an upper bound on the generalization error but in $\alpha$-R\'enyi divergence. We denote $P^0$ the generating distribution of any $(X_i,Y_i)$ and $P_\theta$ the distribution characterizing the model. Then, for any $\alpha \in (0,1)$:
   \begin{align*}
 \mathbb{E} \bigg[ \int D_{\alpha}( P_\theta,P^0 ) \tilde{\pi}_{n,\alpha}(d\theta) \bigg] \leq \inf_{q \in \mathcal{F}_{S,L,{D}}} \bigg\{ \frac{\alpha}{1-\alpha} \int \textnormal{KL}(P^0,P_\theta) q({\rm d} \theta)
 + \frac{\textnormal{KL}(q\|\pi)}{n(1-\alpha)}  \bigg\}.
 \end{align*}
Moreover, the $\alpha$-R\'enyi divergence is equal to $D_{\alpha}( P_\theta,P^0 )=\frac{\alpha}{2\sigma^2}\|f_\theta-f_0\|_2^2$ and the KL divergence is $\textnormal{KL}( P^0\|P_\theta )=\frac{1}{2\sigma^2}\|f_\theta-f_0\|_2^2$, and for any $\theta^*$, $\|f_\theta-f_0\|_2^2 \leq 2 \|f_\theta-f_{\theta^*}\|_2^2 + 2 \|f_{\theta^*}-f_0\|_2^2$. Hence, for any $\theta^* \in \Theta_{S,L,{D}}$:
\begin{align*}
\mathbb{E} \bigg[ &\int \frac{\alpha}{2\sigma^2}\|f_\theta-f_0\|_2^2 \tilde{\pi}_{n,\alpha}(d\theta) \bigg] \\
&\leq \frac{\alpha}{1-\alpha} \frac{2}{2\sigma^2}\|f_{\theta^*}-f_0\|_2^2+ \inf_{q \in \mathcal{F}_{S,L,{D}}} \bigg\{ \frac{\alpha}{1-\alpha} \int \frac{2}{2\sigma^2}\|f_{\theta}-f_{\theta^*}\|_2^2 q({\rm d} \theta)
+ \frac{\textnormal{KL}(q\|\pi)}{n(1-\alpha)}  \bigg\},
\end{align*}
i.e. for any $\theta^* \in \Theta_{S,L,{D}}$,
\begin{align*}
\mathbb{E} \bigg[ &\int \|f_\theta-f_0\|_2^2 \tilde{\pi}_{n,\alpha}(d\theta) \bigg] \\
&\leq \frac{2}{1-\alpha}\|f_{\theta^*}-f_0\|_2^2+ \inf_{q \in \mathcal{F}_{S,L,{D}}} \bigg\{ \frac{2}{1-\alpha} \int \|f_{\theta}-f_{\theta^*}\|_2^2 q({\rm d} \theta)
+ \frac{2\sigma^2}{\alpha} \frac{\textnormal{KL}(q\|\pi)}{n(1-\alpha)}  \bigg\}.
\end{align*}
 
From now on, the rest of the proof consists in finding a distribution $q_n^* \in \mathcal{F}_{S,L,{D}}$ that satisfies for $\theta^*=\argmin_{\theta \in \Theta_{S,L,{D}}} \|f_\theta-f_0\|_2$ the extended prior mass condition, i.e. that satisfies both:
\begin{equation}
\label{KLcond1}
    \int \|f_{\theta}-f_{\theta^*}\|_2^2 q_n^*({\rm d} \theta) \leq r_n
\end{equation}
and
\begin{equation}
\label{KLcond2}
    \textnormal{KL}(q_n^*\|\pi) \leq n r_n 
\end{equation}
with $r_n = \frac{SL}{n}\log(BD) + \frac{S}{n}\log(BL(D+1)^2) + \frac{S}{2n} \log\bigg(\frac{4n}{S}\bigg\{ 3 + (d+2)^2 L^2\bigg\}\bigg)$ that is smaller than $r_n^{S,L,D}$ as $3+(x+2)^2 L^2\leq 10x^2L^2$ for $x\geq1$ and $L\geq3$. This will lead to:
\begin{equation*}
 \mathbb{E} \bigg[ \int \|f_\theta-f_0\|_2^2 \tilde{\pi}_{n,\alpha}(d\theta) \bigg] \leq \frac{2}{1-\alpha} \inf_{\theta^* \in \Theta_{S,L,{D}}} \|f_{\theta^*}-f_0\|_2^2 + \frac{2}{1-\alpha} \bigg(1+\frac{\sigma^2}{\alpha}\bigg)r_n^{S,L,D} .
\end{equation*}

\vspace{0.5cm}

\underline{\textit{Second step}}: \underline{\textit{we prove Inequality \eqref{KLcond1}}}

\vspace{0.5cm}

To begin with, we define the loss of the $\ell^{th}$ layer of the neural network $f_\theta$:
$$
r_\ell(\theta) = \sup_{x\in[-1,1]^d} \sup_{1\leq i\leq D} |f_\theta^\ell(x)_i - f_{\theta^*}^\ell(x)_i| 
$$
where $f_\theta^\ell$s are defined as the partial networks:
$$
\begin{cases}
f_\theta^0(x):=x, \\
f_\theta^\ell(x):= \rho(A_\ell f_\theta^{\ell-1}(x) + b_\ell) \hspace{0.3cm} \text{for} \hspace{0.1cm} \ell=1,...,L.
\end{cases}
$$
We also define the loss of the output layer:
$$
r_\ell(\theta) = \sup_{x\in[-1,1]^d}  |f_\theta^L(x) - f_{\theta^*}^L(x)| = \sup_{x\in[-1,1]^d}  |f_\theta(x) - f_{\theta^*}(x)|.
$$

We will prove by induction that for any $\ell=1,...,L$:
\begin{equation*}
    r_\ell(\theta) \leq (BD)^{\ell}\bigg(d+1+\frac{1}{BD-1}\bigg) \sum_{u=1}^\ell \tilde{A}_u + \sum_{u=1}^\ell (BD)^{\ell-u} \tilde{b}_u 
\end{equation*}
where $\tilde{A}_u=\sup_{i,j} |A_{u,i,j}-A_{u,i,j}^*|$ and $\tilde{b}_u=\sup_{j} |b_{u,j}-b_{u,j}^*|$. To do so, we will also prove by induction that:
\begin{equation*}
    c_\ell \leq B^\ell D^{\ell-1}\bigg(d+1+\frac{1}{BD-1}\bigg)
\end{equation*}
where 
$$
\begin{cases}
c_\ell = \sup_{x\in[-1,1]^d} \sup_{1\leq i\leq D} |f_{\theta^*}^\ell(x)_i|  \hspace{0.3cm} \text{for} \hspace{0.1cm} \ell=1,...,L, \\
c_L = \sup_{x\in[-1,1]^d} |f_{\theta^*}(x)|,
\end{cases} 
$$
using the formula:
\begin{equation}
\label{recurrence}
x_n \leq u_n x_{n-1} + v_n \implies x_n \leq \sum_{i=2}^n \bigg( \prod_{j=i+1}^n u_j \bigg) v_i + \bigg( \prod_{j=2}^n u_j \bigg) x_1
\end{equation}
for any $n\geq2$ with the convention $\prod_{j=n+1}^n u_j=1$.

Indeed, we have according to Assumption \ref{asm1}:
\begin{itemize}
    \item Initialization:
\begin{align*}
    c_1 & = \sup_{x\in[-1,1]^d} \sup_{1\leq i\leq D} |f_{\theta^*}^1(x)_i| \\
    & \leq \sup_{x\in[-1,1]^d} \sup_{1\leq i\leq D} \bigg|\sum_{j=1}^d A_{1ij}^*x_j+b_{1i}^*\bigg| \\
    & \leq \sup_{x\in[-1,1]^d} \sup_{1\leq i\leq D} \bigg\{\sum_{j=1}^d |A_{1ij}^*|\cdot|x_j|+|b_{1i}^*|\bigg\} \\
    & \leq d\cdot B\cdot 1 + B \\
    & = (d+1)B.
\end{align*}
    \item For any layer $\ell$:
    \begin{align*}
    c_\ell & \leq \sup_{x\in[-1,1]^d} \sup_{1\leq i\leq D} \bigg|\sum_{j=1}^D A_{\ell ij}^*f_{\theta^*}^{\ell-1}(x)_j+b_{\ell i}^*\bigg| \\
    & \leq \sup_{x\in[-1,1]^d} \sup_{1\leq i\leq D} \bigg\{\sum_{j=1}^D |A_{\ell ij}^*|\cdot|f_{\theta^*}^{\ell-1}(x)_j|+|b_{\ell i}^*|\bigg\} \\
    & \leq D\cdot B\cdot c_{\ell-1} + B.
\end{align*}
    \item Hence, using Formula \eqref{recurrence}, we get:
    \begin{align*}
    c_\ell & \leq  \sum_{u=2}^\ell \bigg( \prod_{v=u+1}^\ell DB \bigg) B + \bigg( \prod_{v=2}^\ell BD \bigg) c_1 \\
    & \leq B \sum_{u=2}^\ell (DB)^{\ell-u} + (BD)^{\ell-1}(d+1)B \\
    & = B \sum_{u=0}^{\ell-2} (DB)^{u} + (d+1)D^{\ell-1}B^\ell \\
    & = B \frac{(BD)^{\ell-1}-1}{BD-1} + (d+1)D^{\ell-1}B^\ell \\
    & \leq B^\ell D^{\ell-1}\bigg(d+1+\frac{1}{BD-1}\bigg) .
\end{align*}
\end{itemize}
Let us now come back to finding an upper bound on losses of the partial networks $f_\theta^\ell$s. As previously, we have:
\begin{itemize}
    \item Initialization:
\begin{align*}
    r_1(\theta) & = \sup_{x\in[-1,1]^d} \sup_{1\leq i\leq D} |f_{\theta^*}^1(x)_i-f_{\theta}^1(x)_i| \\
    & \leq \sup_{x\in[-1,1]^d} \sup_{1\leq i\leq D} \bigg\{\sum_{j=1}^d |A_{1ij}-A_{1ij}^*|\cdot|x_j|+|b_{1i}-b_{1i}^*|\bigg\} \\
    & \leq d\cdot \tilde{A}_1 + \tilde{b}_1.
\end{align*}
    \item For any layer $\ell$:
    \begin{align*}
    r_\ell(\theta) & \leq \sup_{x\in[-1,1]^d} \sup_{1\leq i\leq D} \bigg\{\sum_{j=1}^D |A_{\ell ij}f_{\theta}^{\ell-1}(x)_j-A_{\ell ij}^*f_{\theta^*}^{\ell-1}(x)_j|+|b_{\ell i}-b_{\ell i}^*|\bigg\} \\
    & \leq \sup_{x\in[-1,1]^d} \sup_{1\leq i\leq D} \bigg\{\sum_{j=1}^D \bigg[ |A_{\ell ij}-A_{\ell ij}^*|\cdot |f_{\theta^*}^{\ell-1}(x)_j|+|A_{\ell ij}|\cdot |f_{\theta^*}^{\ell-1}(x)_j-f_{\theta}^{\ell-1}(x)_j|\bigg] \\
    & \quad \quad \quad \quad \quad \quad \quad \quad \quad +|b_{\ell i}-b_{\ell i}^*|\bigg\} \\
    & \leq D c_{\ell-1}\tilde{A}_\ell + BD r_{\ell-1}(\theta) + \tilde{b}_\ell \\
    & \leq BD r_{\ell-1}(\theta) + \tilde{A}_\ell B^{\ell-1} D^{\ell-1} \bigg(d+1+\frac{1}{BD-1}\bigg) + \tilde{b}_\ell.
\end{align*}
    \item Finally, using Formula \eqref{recurrence}:
    \begin{align*}
    r_\ell(\theta) & \leq  \sum_{u=2}^\ell \bigg( \prod_{v=u+1}^\ell BD \bigg) \bigg( \tilde{A}_u (BD)^{u-1}\bigg\{d+1+\frac{1}{BD-1}\bigg\}+\tilde{b}_u \bigg) + \bigg( \prod_{v=2}^\ell BD \bigg) r_1(\theta) \\
    & = \sum_{u=2}^\ell (BD)^{\ell-u} \tilde{A}_u(BD)^{u-1}\bigg(d+1+\frac{1}{BD-1}\bigg) + \sum_{u=2}^\ell (BD)^{\ell-u} \tilde{b}_u + (BD)^{\ell-1}r_1(\theta) \\
    & \leq \bigg(d+1+\frac{1}{BD-1}\bigg) \sum_{u=2}^\ell (BD)^{\ell-1} \tilde{A}_u + \sum_{u=2}^\ell (BD)^{\ell-u} \tilde{b}_u + (BD)^{\ell-1} d\tilde{A}_1 \\
    & \quad \quad \quad \quad \quad \quad \quad \quad \quad \quad \quad \quad \quad \quad \quad \quad \quad \quad \quad \quad \quad \quad \quad \quad \quad \quad \quad + (BD)^{\ell-1}\tilde{b}_1 \\
    & \leq (BD)^{\ell-1} \bigg(d+1+\frac{1}{BD-1}\bigg) \sum_{u=1}^\ell \tilde{A}_u + \sum_{u=1}^\ell (BD)^{\ell-u} \tilde{b}_u .
\end{align*}
\end{itemize}

Then, for any distribution $q$:
\begin{align*}
    \int \|f_{\theta}&-f_{\theta^*}\|_2^2 q({\rm d} \theta) \leq \int \|f_{\theta}-f_{\theta^*}\|_{\infty}^2 q({\rm d} \theta) = \int r_L(\theta)^2 q({\rm d} \theta) \\ 
    & \leq \int 2 (BD)^{2L-2}\bigg(d+1+\frac{1}{BD-1}\bigg)^2 \bigg( \sum_{\ell=1}^L \tilde{A}_\ell \bigg)^2 q({\rm d} \theta) + \int 2 \bigg( \sum_{\ell=1}^L (BD)^{L-\ell} \tilde{b}_u \bigg)^2 q({\rm d} \theta) \\
    & = 2 (BD)^{2L-2}\bigg(d+1+\frac{1}{BD-1}\bigg)^2 \bigg( \int \sum_{\ell=1}^L \tilde{A}_\ell^2 q({\rm d} \theta) + 2 \int \sum_{\ell=1}^L \sum_{k=1}^{\ell-1} \tilde{A}_\ell \tilde{A}_k q({\rm d} \theta) \bigg) \\
    & \quad + 2 \bigg( \int \sum_{\ell=1}^L (BD)^{2(L-\ell)} \tilde{b}_l^2 q({\rm d} \theta) + 2 \int \sum_{\ell=1}^L \sum_{k=1}^{\ell-1} (BD)^{L-\ell} (BD)^{L-k} \tilde{b}_\ell \tilde{b}_k q({\rm d} \theta) \bigg).
\end{align*}

Here, we define $q_n^*(\theta)$ as follows:
$$
\begin{cases}
\gamma_t^* = \mathbb{I}(\theta^*_t\ne 0), \\
\theta_t \sim \gamma_t^* \hspace{0.1cm} \mathcal{U}([\theta_t^*-s_n,\theta^*_t+s_n]) + (1-\gamma_t^*)\delta_{\{0\}} \hspace{0.3cm} \textnormal{for each} \hspace{0.2cm} t=1,...,T.
\end{cases}
$$
with $s_n^2=\frac{S}{4n} (BD)^{-2L} \bigg\{ \bigg(d+1+\frac{1}{BD-1}\bigg)^2 \frac{L^2}{(BD)^2} + \frac{1}{(BD)^2-1} + \frac{2}{(BD-1)^2} \bigg\}^{-1}$. Hence:
$$
\int \tilde{A}_\ell^2 q_n^*({\rm d} \theta) = \int \sup_{i,j} (A_{\ell,i,j}-A_{\ell,i,j}^*)^2 q_n^*({\rm d} A_{\ell,i,j}) \leq s_n^2,
$$
and
\begin{align*}
    \int \tilde{A}_{\ell} \tilde{A}_{k} q_n^*({\rm d} \theta) & = \bigg( \int \sup_{i,j} |A_{\ell,i,j}-A_{\ell,i,j}^*| q_n^*({\rm d} \theta) \bigg) \bigg( \int \sup_{i,j} |A_{k,i,j}-A_{k,i,j}^*| q_n^*({\rm d} \theta) \bigg) \\ 
    & \leq |s_n| \cdot |s_n| = s_n^2 ,
\end{align*}
and similarly, $\int \tilde{b}_\ell^2 q_n^*({\rm d} \theta) \leq s_n^2$ and $\int \tilde{b}_{\ell} \tilde{b}_{k} q_n^*({\rm d} \theta) \leq s_n^2$.

Then
\begin{align*}
    \int & \|f_{\theta}-f_{\theta^*}\|_2^2 q_n^*({\rm d} \theta) \\
    & \leq 2 (BD)^{2L-2}\bigg(d+1+\frac{1}{BD-1}\bigg)^2 \bigg( \int \sum_{\ell=1}^L \tilde{A}_\ell^2 q({\rm d} \theta) + 2 \int \sum_{\ell=1}^L \sum_{k=1}^{\ell-1} \tilde{A}_\ell \tilde{A}_k q({\rm d} \theta) \bigg) \\
    & \quad + 2 \bigg( \int \sum_{\ell=1}^L (BD)^{2(L-\ell)} \tilde{b}_l^2 q({\rm d} \theta) + 2 \int \sum_{\ell=1}^L \sum_{k=1}^{\ell-1} (BD)^{L-\ell} (BD)^{L-k} \tilde{b}_\ell \tilde{b}_k q({\rm d} \theta) \bigg) \\
    & \leq 2 (BD)^{2L-2}\bigg(d+1+\frac{1}{BD-1}\bigg)^2 s_n^2 \bigg( L + 2 \sum_{\ell=0}^{L-1} \ell \bigg) \\
    & \quad + 2 s_n^2 \sum_{\ell=0}^{L-1} (BD)^{2\ell} + 4 s_n^2 \sum_{\ell=1}^L \sum_{k=L-\ell+1}^{L-1} (BD)^{L-\ell} (BD)^{k} \\
    & = 2 (BD)^{2L-2}\bigg(d+1+\frac{1}{BD-1}\bigg)^2 s_n^2 L^2 \\
    & \quad + 2 s_n^2 \frac{(BD)^{2L}-1}{(BD)^2-1} + 4 s_n^2 \sum_{\ell=1}^L \sum_{k=0}^{\ell-2} (BD)^{L-\ell} (BD)^{k} (BD)^{L-\ell+1} \\
    & = 2 s_n^2 (BD)^{2L-2}\bigg(d+1+\frac{1}{BD-1}\bigg)^2 L^2 \\
    & \quad + 2 s_n^2 \frac{(BD)^{2L}-1}{(BD)^2-1} + 4 s_n^2 \sum_{\ell=1}^L (BD)^{L-\ell} \frac{(BD)^{\ell-1}-1}{BD-1} (BD)^{L-\ell+1} \\
    & \leq 2 s_n^2 (BD)^{2L-2}\bigg(d+1+\frac{1}{BD-1}\bigg)^2 L^2 \\
    & \quad + 2 s_n^2 \frac{(BD)^{2L}-1}{(BD)^2-1} + 4 s_n^2 \frac{1}{BD-1} \sum_{\ell=1}^L (BD)^{2L-\ell} \\
    & = 2 s_n^2 (BD)^{2L-2}\bigg(d+1+\frac{1}{BD-1}\bigg)^2 L^2 \\
    & \quad + 2 s_n^2 \frac{(BD)^{2L}-1}{(BD)^2-1} + 4 s_n^2 \frac{1}{BD-1}(BD)^L \frac{(BD)^L-1}{BD-1}  \\
    & \leq 2 s_n^2 (BD)^{2L-2}\bigg(d+1+\frac{1}{BD-1}\bigg)^2 L^2 + 2 s_n^2 \frac{(BD)^{2L}-1}{(BD)^2-1} + 4 s_n^2 \frac{1}{(BD-1)^2}(BD)^{2L}  \\
    & \leq 2 s_n^2 (BD)^{2L} \bigg\{ \bigg(d+1+\frac{1}{BD-1}\bigg)^2 \frac{L^2}{(BD)^2} + \frac{1}{(BD)^2-1} + \frac{2}{(BD-1)^2} \bigg\} \\
    & = \frac{S}{2n} \\
    & \leq r_n
\end{align*}
which proves Equation (\ref{KLcond1}).

\vspace{0.5cm}

\underline{\textit{Third step}}: \underline{\textit{we prove Inequality \eqref{KLcond2}}}

\vspace{0.5cm}

We will use the fact that for any $K$, any $p,p^0 \in [0,1]^K$ such that $\sum_{k=1}^K p_k = \sum_{k=1}^K p^0_k = 1$ and any distributions $Q_k,Q^0_k$ for $k=1,...,K$, we have:
\begin{equation}
\label{Do}
    \mathcal{K}\left( \sum_{k=1}^K p_k^0 Q_k^0 \bigg\| \sum_{k=1}^K p_k Q_k \right) \leq \mathcal{K}(p^0\|p) + \sum_{k=1}^K p_k^0 \mathcal{K}(Q^0_k\|Q_k).
\end{equation}
Please refer to Lemma 6.1 in \cite{cherief2018consistency} for a proof. Then we write $q_n^*$ and $\pi$ as mixtures of independent products of mixtures of two components:
$$
q_n^* = \sum_{\gamma \in \mathcal{S}^S_T} \mathbb{I}(\gamma=\gamma^*) \bigotimes_{t=1}^T \bigg\{ \gamma_t \hspace{0.1cm} \mathcal{U}([l_t,u_t]) + (1-\gamma_t)\delta_{\{0\}} \bigg\}
$$
and
$$
\pi = \sum_{\gamma \in \mathcal{S}^S_T} \binom{T}{S^*}^{-1} \bigotimes_{t=1}^T \bigg\{ \gamma_t \hspace{0.1cm} \mathcal{U}([-B,B]) + (1-\gamma_t)\delta_{\{0\}} \bigg\}
$$
Hence, using Inequality \ref{Do} twice and the additivity of KL for independent distributions:

\begin{align*}
    \textnormal{KL}(q_n^*\|\pi) & \leq \textnormal{KL}\bigg(\{\mathbb{I}(\gamma=\gamma^*)\}_{\gamma \in \mathcal{S}^S_T} \bigg\| \bigg\{\binom{T^*}{S^*}^{-1}\bigg\}_{\gamma \in \mathcal{S}^S_T} \bigg) + \sum_{\gamma \in \mathcal{S}^S_T} \mathbb{I}(\gamma=\gamma^*) \\ 
    & \hspace{0.5cm} \textnormal{KL}\bigg( \bigotimes_{t=1}^{T} \bigg\{ \gamma_t \hspace{0.1cm} \mathcal{U}([l_t,u_t]) + (1-\gamma_t)\delta_{\{0\}} \bigg\} \bigg\| \bigotimes_{t=1}^{T} \bigg\{ \gamma_t \hspace{0.1cm} \mathcal{U}([-B,B]) + (1-\gamma_t)\delta_{\{0\}} \bigg\}\bigg) \\ 
    & = \log\binom{T}{S} + \sum_{t=1}^{T} \textnormal{KL}\bigg( \gamma_t^* \hspace{0.1cm} \mathcal{U}([l_t,u_t]) + (1-\gamma_t^*)\delta_{\{0\}} \bigg\| \gamma_t^* \hspace{0.1cm} \mathcal{U}([-B,B]) + (1-\gamma_t^*)\delta_{\{0\}} \bigg) \\
    & \leq \log\binom{T}{S} + \sum_{t=1}^{T} \gamma_t^* \textnormal{KL}\bigg( \mathcal{U}([l_t,u_t]) \bigg\| \mathcal{U}([-B,B]) \bigg) + \sum_{t=1}^{T} (1-\gamma_t^*) \textnormal{KL}\big( \delta_{\{0\}} \|  \delta_{\{0\}} \big) \\
    & \leq S\log(T) + \sum_{t=1}^{T} \gamma_t^* \log\bigg(\frac{2B}{u_t-l_t}\bigg) \\
    & = S\log(T) + \sum_{t=1}^{T} \gamma_t^* \log\bigg(\frac{2B}{2s_n}\bigg) \\
    & = S\log(T) + S \log(B) + \frac{S}{2} \log\bigg(\frac{1}{s_n^2}\bigg)  \\
    & = S\log(T) + S \log(B) \\
    & \quad + \frac{S}{2} \log\bigg(\frac{4n}{S} (BD)^{2L} \bigg\{ \bigg(d+1+\frac{1}{BD-1}\bigg)^2 L^2 + \frac{1}{(BD)^2-1} + \frac{2}{(BD-1)^2} \bigg\}\bigg) ,
\end{align*}
and hence,
\begin{align*}
    \textnormal{KL}(q_n^*\|\pi) & \leq S\log(T) + S \log(B) \\
    & \quad + \frac{S}{2} \log\bigg(\frac{4n}{S} (BD)^{2L} \bigg\{ \bigg(d+1+\frac{1}{BD-1}\bigg)^2 L^2 + \frac{1}{(BD)^2-1} + \frac{2}{(BD-1)^2} \bigg\}\bigg) \\
    & \leq S\log(L(D+1)^2) + S \log(B) + LS\log(BD) \\
    & \quad + \frac{S}{2} \log\bigg(\frac{4n}{S}\bigg\{ \bigg(d+1+\frac{1}{BD-1}\bigg)^2 L^2 + \frac{1}{(BD)^2-1}  + \frac{2}{(BD-1)^2} \bigg\}\bigg) \\
    & \leq n r_n,
\end{align*}

which ends the proof.

\section{Proof of Corollary \ref{cor-concentration}}
\label{app:proof-cor-3}
Corollary \ref{cor-concentration} is a direct consequence of Theorem \ref{thm-general}, and we just need to find an upper bound on $\inf_{\theta^* \in \Theta_{S,L,D}} \|f_{\theta^*}-f_0\|_\infty^2$ and $r_n^{S,L,D}$. Indeed, according to Theorem \ref{thm-general}:
\begin{align}
\label{ProofIne}
\mathbb{E} \bigg[ \int \|f_\theta-f_0\|_2^2 \tilde{\pi}_{n,\alpha}(d\theta) \bigg] \leq \frac{2}{1-\alpha} \inf_{\theta^* \in \Theta_{S,L,D}} \|f_{\theta^*}-f_0\|_\infty^2 + \frac{2}{1-\alpha} \bigg(1+\frac{\sigma^2}{\alpha}\bigg) r_n.
\end{align}
We directly use the rate $r_n$ in the proof of Theorem \ref{thm-general} rather than $r_n^{S,L,D}$.

\vspace{0.2cm}
Let us assume that $f_0$ is $\beta$-H\"older smooth with $0<\beta<d$. Then according to Lemma 5.1 in \cite{Rockova2018}, we have for some positive constant $C_D$ independent of $n$ (see Theorem 6.1 in \cite{Rockova2018}) a neural network with architecture :
$$
L= 8 + (\lfloor\log_2n\rfloor + 5)(1 + \lceil\log_2 d\rceil),
$$
$$
D = C_D \lfloor n^{\frac{d}{2\beta+d}}/\log n \rfloor,
$$
$$
S \leq 94 d^2 (\beta+1)^{2d} D (L + \lceil\log_2 d\rceil),
$$
with an error $\|f-f_0\|_\infty$ that is at most a constant multiple of $\frac{D}{n}+D^{-\beta/d}\leq C_D n^{\frac{-2\beta}{2\beta+d}}/\log n + C_D^{-\beta/d} n^{\frac{-\beta}{2\beta+d}} \log^{\beta/d} n \leq (C_D/\log n + C_D^{-\beta/d} \log n) n^{\frac{-\beta}{2\beta+d}}$, which gives an upper bound on the first term of the right-hand-side of Inequality \ref{ProofIne} of order $n^{\frac{-2\beta}{2\beta+d}} \log^2 n$.
 
In the same time, we have for some constants $C,C'$ that do not depend on $n$:
\begin{align*}
    r_n & \leq \frac{SL}{n}\log(BD) + \frac{S}{n}\log(2BL(D+1)^2) + \frac{S}{2n} \log\bigg(\frac{4n}{S}\bigg\{ 3+ (d+2)^2 L^2\bigg\}\bigg) \\
    & \leq C \bigg( \frac{DL^2}{n} \log D + \frac{DL}{n} \log(LD) + \frac{DL}{n} \log n \bigg) \\
    & \leq C' \frac{n^{\frac{d}{2\beta+d}}}{n} \log^2 n = C' n^{\frac{-2\beta}{2\beta+d}} \log^2 n .
\end{align*}

Then the tempered posterior distribution $\pi_{n,\alpha}$ concentrates at the
minimax rate $r_n = n^{\frac{-2\beta}{2\beta+d}}$ up to a (squared) logarithmic factor for the expected $L_2$-distance in the sense that:
\begin{equation*}
\label{thm-concentration-holder}
{\pi}_{n,\alpha}\bigg(\theta\in\Theta_{S,L,D} \hspace{0.1cm} \big/ \hspace{0.1cm} \|f_\theta-f_0\|_2^2>M_n n^{\frac{-2\beta}{2\beta+d}} \log^2 n \bigg) \xrightarrow[n\rightarrow +\infty]{} 0.
\end{equation*}
in probability as $n\rightarrow+\infty$ for any $M_n\rightarrow +\infty$.

\section{Proof of Theorem \ref{thm-concentration}}
\label{app:proof-thm-1}
We could prove Theorem \ref{thm-concentration} using the prior mass condition \eqref{prior-mass-condition} but we will use instead the same proof than for Theorem \ref{thm-general}. Indeed, we can easily show that for any $\theta^* \in \Theta_{S,L,{D}}$,
\begin{equation*}
\mathbb{E} \bigg[ \int \|f_\theta-f_0\|_2^2 \pi_{n,\alpha}(d\theta) \bigg] \leq \frac{2}{1-\alpha}\|f_{\theta^*}-f_0\|_2^2+ \inf_{q} \bigg\{ \frac{2}{1-\alpha} \int \|f_{\theta}-f_{\theta^*}\|_2^2 q({\rm d} \theta)
+ \frac{2\sigma^2}{\alpha} \frac{\textnormal{KL}(q\|\pi)}{n(1-\alpha)}  \bigg\}
\end{equation*}
where the infimum is taken over all the probability distributions on $\Theta_{S,L,D}$. We have:
\begin{align*}
    \inf_{q} \bigg\{ \frac{2}{1-\alpha} \int \|f_{\theta}-f_{\theta^*}\|_2^2 & q({\rm d} \theta) + \frac{2\sigma^2}{\alpha} \frac{\textnormal{KL}(q\|\pi)}{n(1-\alpha)}  \bigg\} \\
    &\leq \inf_{q \in \mathcal{F}_{S,L,D}} \bigg\{ \frac{2}{1-\alpha} \int \|f_{\theta}-f_{\theta^*}\|_2^2 q({\rm d} \theta) + \frac{2\sigma^2}{\alpha} \frac{\textnormal{KL}(q\|\pi)}{n(1-\alpha)} \bigg\} \\
    &\leq \frac{2}{1-\alpha} \bigg(1+\frac{\sigma^2}{\alpha}\bigg) r_n^{S,L,D},
\end{align*}
which implies
\begin{align*}
    \mathbb{E} \bigg[ \int \|f_\theta-f_0\|_2^2 \tilde{\pi}_{n,\alpha}(d\theta) \bigg] &\leq \frac{2}{1-\alpha} \inf_{\theta^* \in \Theta_{S,L,D}} \|f_{\theta^*}-f_0\|_2^2 + \frac{2}{1-\alpha} \bigg(1+\frac{\sigma^2}{\alpha}\bigg) r_n^{S,L,D} \\
    &\leq \frac{2}{1-\alpha} \inf_{\theta^* \in \Theta_{S,L,D}} \|f_{\theta^*}-f_0\|_\infty^2 + \frac{2}{1-\alpha} \bigg(1+\frac{\sigma^2}{\alpha}\bigg) r_n^{S,L,D} .
\end{align*}

The rest of the proof follows the same lines than the one of Corollary \ref{cor-concentration}.

\section{Proof of Theorem \ref{thm-opti}}
\label{app:proof-thm-4}
First, we need Donsker and Varadhan's variational formula. Refer to Lemma 1.1.3. in \cite{CatoniThermo} for a proof.

\begin{thm}
\label{thm-dv}
For any probability $\lambda$ on some measurable space $(\textbf{E},\mathcal{E})$ and any measurable function
  $h: \textbf{E} \rightarrow \mathbb{R}$ such that $\int{\rm e}^h  \rm{d}\lambda < \infty$,
  \begin{equation*}
    \log\int {\rm e}^h \mathrm{d}\lambda = \underset{q}{\sup} \bigg\{ \int h \mathrm{d}q - \textrm{KL}(q,\lambda) \bigg\},
  \end{equation*}
  where the supremum is taken over all probability distributions over $\textbf{E}$ and with the convention
  $\infty-\infty =-\infty$. Moreover, if $h$ is upper-bounded on the
  support of $\lambda$, then the supremum is reached by the distribution of the form:
  \begin{equation*}
    \lambda_h(d\beta) =
    \frac{{\rm e}^{h(\beta)} }{\int{\rm e}^h \mathrm{d}\lambda} \lambda(\mathrm{d}\beta).
  \end{equation*}
\end{thm}

Let us come back to the proof of Theorem \ref{thm-opti}. Here, we can not directly use Theorem 2.6 in \cite{Tempered}. Thus we begin from scratch.
For any $\alpha \in (0,1)$ and $\theta \in\Theta_{S,L,D}$, using the definition of R\'enyi divergence and $D_\alpha(P^{\otimes n},R^{\otimes n})=nD_\alpha(P,R)$ as data are i.i.d.\:
$$
\mathbb{E}\bigg[ \exp\bigg(-\alpha r_{n}(P_\theta,P^0) + (1-\alpha)n D_\alpha(P_\theta,P^0)\bigg) \bigg] = 1
$$
where $r_n( P_\theta,P^0 ) = \frac{1}{2\sigma^2} \sum_{i=1}^n \{ (Y_i-f_\theta(X_i))^2 - (Y_i-f_0(X_i))^2 \}$ is the negative log-likelihood ratio.
Then we integrate and use Fubini's theorem, 
$$
\mathbb{E}\bigg[ \int \exp\bigg(-\alpha r_{n}(P_\theta,P^0) + (1-\alpha)n D_\alpha(P_\theta,P^0) \bigg) \pi(d\theta) \bigg] = 1.
$$
According to Theorem \ref{thm-dv},
\begin{multline*}
    \mathbb{E}\bigg[ \exp\bigg( \sup_{q} \bigg\{ \int \bigg( -\alpha r_{n}(P_\theta,P^0) + (1-\alpha)n D_\alpha(P_\theta,P^0) \bigg) q(d\theta) - \textrm{KL}(q||\pi) \bigg\} \bigg) \bigg] = 1
\end{multline*}
where the supremum is taken over all probability distributions over $\Theta_{S,L,D}$.
Then, using Jensen's inequality, 
$$ 
\mathbb{E}\bigg[ \sup_{q} \bigg\{ \int \bigg( -\alpha r_{n}(P_\theta,P^0) + (1-\alpha)n D_\alpha(P_\theta,P^0) \bigg) q(d\theta) - \textrm{KL}(q||\pi) \bigg\} \bigg] \leq 0,
$$
and then, 
$$
\mathbb{E}\bigg[ \int \bigg( -\alpha r_{n}(P_\theta,P^0) + (1-\alpha)n D_\alpha(P_\theta,P^0) \bigg) \tilde{\pi}_{n,\alpha}^k(d\theta) - \textrm{KL}(\tilde{\pi}^k_{n,\alpha}||\pi) \bigg] \leq 0.
$$
We rearrange terms:
\begin{equation*}
\mathbb{E}\bigg[ \int D_\alpha(P_\theta,P^0) \tilde{\pi}_{n,\alpha}^k(d\theta) \bigg] \leq 
\mathbb{E}\bigg[ \frac{\alpha}{1-\alpha} \int \frac{r_{n}(P_\theta,P^0)}{n} \tilde{\pi}_{n,\alpha}^k(d\theta) + \frac{\textrm{KL}(\tilde{\pi}^k_{n,\alpha}||\pi)}{n(1-\alpha)} \bigg],
\end{equation*}
that we can write:
\begin{align*}
    \mathbb{E} \bigg[ \int D_{\alpha}( P_\theta,P^0 ) \tilde{\pi}_{n,\alpha}^k(d\theta) \bigg] & \leq \mathbb{E}\bigg[ \frac{\alpha}{1-\alpha} \int \frac{r_{n}(P_\theta,P^0)}{n} \tilde{\pi}_{n,\alpha}(d\theta) + \frac{\textrm{KL}(\tilde{\pi}_{n,\alpha}||\pi)}{n(1-\alpha)} \bigg] \\
    & \quad + \mathbb{E}\bigg[ \frac{\alpha}{1-\alpha} \int \frac{r_{n}(P_\theta,P^0)}{n} \tilde{\pi}_{n,\alpha}^k(d\theta) + \frac{\textrm{KL}(\tilde{\pi}^k_{n,\alpha}||\pi)}{n(1-\alpha)} \bigg] \\
    & \quad \quad - \mathbb{E}\bigg[ \frac{\alpha}{1-\alpha} \int \frac{r_{n}(P_\theta,P^0)}{n} \tilde{\pi}_{n,\alpha}(d\theta) + \frac{\textrm{KL}(\tilde{\pi}_{n,\alpha}||\pi)}{n(1-\alpha)} \bigg] .
\end{align*}
Let us precise that $
\mathbb{E}\bigg[\frac{r_n( P_\theta,P^0 )}{n}\bigg] = \textnormal{KL}(P^0||P_\theta) = \frac{\|f_0-f_\theta\|^2_2}{2\sigma^2} ,
$ and:
$$
\mathcal{L}_n(q) = - \frac{\alpha}{2\sigma^2} \sum_{i=1}^n \int (Y_i-f_\theta(X_i))^2 q(d\theta) - \textrm{KL}(q\|\pi) \hspace{0.5cm} \textnormal{up to a constant.}
$$
Then:
\begin{equation*}
\mathbb{E} \bigg[ \int D_{\alpha}( P_\theta,P^0 ) \tilde{\pi}_{n,\alpha}^k(d\theta) \bigg] \leq \mathbb{E}\bigg[ \frac{\alpha}{1-\alpha} \int \frac{r_{n}(P_\theta,P^0)}{n} \tilde{\pi}_{n,\alpha}(d\theta) + \frac{\textrm{KL}(\tilde{\pi}_{n,\alpha}||\pi)}{n(1-\alpha)} \bigg] + \frac{\mathbb{E}[\mathcal{L}^*_n-\mathcal{L}_n^k]}{n(1-\alpha)} .
\end{equation*}
We conclude by interverting the infimum and the expectation and the same inequalities than in Theorem \ref{thm-general}:
\begin{align*}
    \mathbb{E}\bigg[ \frac{\alpha}{1-\alpha} \int \frac{r_{n}(P_\theta,P^0)}{n} \tilde{\pi}_{n,\alpha}&(d\theta)  + \frac{\textnormal{KL}(\tilde{\pi}_{n,\alpha}\|\pi)}{n(1-\alpha)}  \bigg] \\
    & = \mathbb{E} \bigg[ \inf_{q \in \mathcal{F}_{S,L,{D}}} \bigg\{ \frac{\alpha}{1-\alpha} \int \frac{r_n( P_\theta,P^0 )}{n} q({\rm d} \theta) + \frac{\textnormal{KL}(q\|\pi)}{n(1-\alpha)}  \bigg\} \bigg] \\
    & \leq \inf_{q \in \mathcal{F}_{S,L,{D}}} \bigg\{ \mathbb{E} \bigg[ \frac{\alpha}{1-\alpha} \int \frac{r_n( P_\theta,P^0 )}{n} q({\rm d} \theta) + \frac{\textnormal{KL}(q\|\pi)}{n(1-\alpha)} \bigg] \bigg\} \\
    & \leq \frac{\alpha}{1-\alpha} \frac{2}{2\sigma^2} \inf_{\theta^* \in \Theta_{S,L,{D}}} \|f_{\theta^*}-f_0\|_2^2 + \frac{\alpha}{2\sigma^2} \frac{2}{1-\alpha} \bigg(1+\frac{\sigma^2}{\alpha}\bigg)r_n^{S,L,D}.
\end{align*}

\section{Proof of Theorem \ref{thm-model-selection}}
\label{app:proof-thm-5}
We start from the last inequality obtained in the proof of Theorem 3 in \cite{cherief2018consistency2} that provides an upper bound in $\alpha$-R\'enyi divergence for the ELBO model selection framework. We still denote $P^0$ the generating distribution and $P_\theta$ the distribution characterizing the model. Then, for any $\alpha \in (0,1)$:
\begin{multline*}
    \mathbb{E}\bigg[ \int D_\alpha(P_\theta,P^0) \tilde{\pi}_{n,\alpha}^{\hat{S},\hat{L},\hat{D}}(d\theta) \bigg] \\
    \leq \inf_{S,L,D} \bigg\{ \inf_{q \in \mathcal{F}_{S,L,D}} \bigg\{ \frac{\alpha}{1-\alpha} \int \textrm{KL}(P^0,P_{\theta_{S,L,D}}) q(d\theta_{S,L,D}) + \frac{\textrm{KL}(q,\Pi^{S,L,D})}{n(1-\alpha)} \bigg\} +  \frac{\log(\frac{1}{\pi_{S,L,D}})}{n(1-\alpha)} \bigg\}
\end{multline*}
where $\Pi^{S,L,D}$ denotes the prior over the parameter set $\Theta_{S,L,D}$ and $\pi_{S,L,D}$ the prior belief over model $(S,L,D)$.

As for the proof of Theorem \ref{thm-general}, for any $S,L,D$ and any $\theta^* \in \Theta_{S,L,{D}}$:
\begin{align*}
\mathbb{E} &\bigg[ \int \frac{\alpha}{2\sigma^2}\|f_\theta-f_0\|_2^2 \tilde{\pi}_{n,\alpha}^{\hat{S},\hat{L},\hat{D}}(d\theta) \bigg] \\
&\leq \frac{\alpha}{1-\alpha} \frac{2}{2\sigma^2}\|f_{\theta^*}-f_0\|_2^2+ \inf_{q \in \mathcal{F}_{S,L,{D}}} \bigg\{ \frac{\alpha}{1-\alpha} \int \frac{2}{2\sigma^2}\|f_{\theta}-f_{\theta^*}\|_2^2 q(d\theta)
&+ \frac{\textrm{KL}(q,\Pi^{S,L,D})}{n(1-\alpha)} \bigg\} \\
& & + \frac{\log(\frac{1}{\pi_{S,L,D}})}{n(1-\alpha)} ,
\end{align*}
and then for any $S,L,D$ and any $\theta^* \in \Theta_{S,L,{D}}$,
\begin{align*}
\mathbb{E} \bigg[ \int \|f_\theta-f_0\|_2^2 \tilde{\pi}_{n,\alpha}^{\hat{S},\hat{L},\hat{D}}(d\theta) \bigg]
\leq \frac{2}{1-\alpha} \|f_{\theta^*}-f_0\|_2^2+ \frac{2}{1-\alpha} \bigg(1+&\frac{\sigma^2}{\alpha}\bigg) r_n^{S,L,D} \\
& + \frac{2\sigma^2}{\alpha(1-\alpha)} \frac{\log(\frac{1}{\pi_{S,L,D}})}{n} ,
\end{align*} 
which finally leads to Theorem \ref{thm-model-selection}.

\section{Result for sparse Gaussian approximations}
\label{app:Gauss}
In this appendix, we consider non-bounded parameter sets $\Theta_{S,L,D}$ and Gaussians instead of uniform distributions in spike-and-slab priors on $\theta \in \Theta_{S,L,D}$:
$$
\begin{cases}
\gamma \sim \mathcal{U}(\mathcal{S}^S_T), \\
\theta_t|\gamma_t \sim \gamma_t \hspace{0.1cm} \mathcal{N}(0,1) + (1-\gamma_t)\delta_{\{0\}}, \hspace{0.2cm} t=1,...,T
\end{cases}
$$
and Gaussian-based sparse spike-and-slab approximations:
$$
\begin{cases}
\gamma \sim \pi_{\gamma}, \\
\theta_t|\gamma_t \sim \gamma_t \hspace{0.1cm} \mathcal{N}(m_t,s_n^2) + (1-\gamma_t)\delta_{\{0\}} \hspace{0.3cm} \textnormal{for each} \hspace{0.2cm} t=1,...,T.
\end{cases}
$$
The following theorem states that using Gaussians instead of uniform distributions still leads to consistency with the same rate of convergence. Note that the infimum in the RHS of the inequality is taken over a bounded neural network model.

\begin{thm}
 \label{thm-Gauss}
Let us introduce the sets $\Theta_{S,L,D}^B$ that contain the neural network parameters upper bounded by $B$ (in $L_\infty$-norm). Then for any $\alpha\in(0,1)$, for any $B\geq 2$,
\begin{equation*}
\label{Gauss-thm-fixed}
 \mathbb{E} \bigg[ \int \|f_\theta-f_0\|_2^2 \tilde{\pi}_{n,\alpha}(d\theta) \bigg] \leq \frac{2}{1-\alpha} \inf_{\theta^* \in \Theta_{S,L,D}^{B}} \|f_{\theta^*}-f_0\|_2^2 + \frac{2}{1-\alpha} \bigg(1+\frac{\sigma^2}{\alpha}\bigg) r_n^{S,L,D}
\end{equation*}
with 
$$
r_n^{S,L,D}= \frac{SL}{n}\log(2BD) + \frac{S}{4n}\bigg(12\log(LD) + B^2\bigg) + \frac{S}{n} \log\bigg(11d\max(\frac{n}{S},1)\bigg) .
$$
\end{thm}

\begin{proof}

The proof follows the same structure than for Theorem \ref{thm-general}. We fix $B\geq 2$.

\vspace{0.2cm}

\underline{\textit{First step}}: \underline{\textit{we obtain the general inequality}}

\vspace{0.2cm}

We can directly write for any $\theta^* \in \Theta_{S,L,{D}}$,
\begin{align*}
\mathbb{E} \bigg[ &\int \|f_\theta-f_0\|_2^2 \tilde{\pi}_{n,\alpha}(d\theta) \bigg] \\
&\leq \frac{2}{1-\alpha}\|f_{\theta^*}-f_0\|_2^2+ \inf_{q \in \mathcal{F}_{S,L,{D}}} \bigg\{ \frac{2}{1-\alpha} \int \|f_{\theta}-f_{\theta^*}\|_2^2 q({\rm d} \theta)
+ \frac{2\sigma^2}{\alpha} \frac{\textnormal{KL}(q\|\pi)}{n(1-\alpha)}  \bigg\}.
\end{align*}
We define $\theta^*=\argmin_{\theta \in \Theta_{S,L,{D}}^{B}} \|f_\theta-f_0\|_2$. Again, the rest of the proof consists in finding a distribution $q_n^* \in \mathcal{F}_{S,L,{D}}$ that satisfies the extended prior mass condition:
\begin{equation}
\label{KLcond1Gauss}
    \int \|f_{\theta}-f_{\theta^*}\|_2^2 q_n^*({\rm d} \theta) \leq r_n 
\end{equation}
and
\begin{equation}
\label{KLcond2Gauss}
    \textnormal{KL}(q_n^*\|\pi) \leq n r_n 
\end{equation}
with $r_n = \frac{SL}{n}\log(2BD) + \frac{S}{n}\log(L(D+1)^2) + \frac{S \log\log(3D)}{n} + \frac{SB^2}{4n} + \frac{S}{2n} \log\bigg(\frac{16n}{S}\bigg\{ 3 + (d+2)^2 \bigg\}\bigg) \leq r_n^{S,L,D}$ as $3+(x+2)^2 \leq 7x^2$ for $x\geq1$.

\vspace{0.5cm}

\underline{\textit{Second step}}: \underline{\textit{we prove Inequality \eqref{KLcond1Gauss}}}

\vspace{0.5cm}

All coefficients of parameter $\theta^*$ are upper bounded by $B$. Hence, we still have:
\begin{equation*}
    c_\ell \leq B^\ell D^{\ell-1}\bigg(d+1+\frac{1}{BD-1}\bigg).
\end{equation*}
However, the upper bound on $r_\ell(\theta)$ is not the same, as $|A_{\ell,i,j}|$ can not be upper bounded by $B$ directly and must be upper bounded by $|A_{\ell,i,j}^*|+\tilde{A_\ell} \leq B +\tilde{A_\ell}$:
\begin{align*}
    r_\ell(\theta) & \leq \sup_{x\in[-1,1]^d} \sup_{1\leq i\leq D} \bigg\{\sum_{j=1}^D \bigg[ |A_{\ell ij}-A_{\ell ij}^*|\cdot |f_{\theta^*}^{\ell-1}(x)_j|+|A_{\ell ij}|\cdot |f_{\theta^*}^{\ell-1}(x)_j-f_{\theta}^{\ell-1}(x)_j|\bigg] \\
    & \quad \quad \quad \quad \quad \quad \quad \quad \quad +|b_{\ell i}-b_{\ell i}^*|\bigg\} \\
    & \leq \sup_{x\in[-1,1]^d} \sup_{1\leq i\leq D} \bigg\{\sum_{j=1}^D \bigg[ |A_{\ell ij}-A_{\ell ij}^*|\cdot |f_{\theta^*}^{\ell-1}(x)_j|+(B+\tilde{A_\ell})\cdot |f_{\theta^*}^{\ell-1}(x)_j-f_{\theta}^{\ell-1}(x)_j| \bigg] \\
    & \quad \quad \quad \quad \quad \quad \quad \quad \quad + |b_{\ell i}-b_{\ell i}^*|\bigg\} \\
    & \leq D c_{\ell-1}\tilde{A}_\ell + (B+\tilde{A_\ell})D r_{\ell-1}(\theta) + \tilde{b}_\ell \\
    & \leq (B+\tilde{A_\ell})D r_{\ell-1}(\theta) + \tilde{A}_\ell B^{\ell-1} D^{\ell-1} \bigg(d+1+\frac{1}{BD-1}\bigg) + \tilde{b}_\ell.
\end{align*}
Then, using Formula \ref{recurrence}:
\begin{align*}
    r_\ell(\theta) & \leq  \sum_{u=2}^\ell \bigg( \prod_{v=u+1}^\ell (B+\tilde{A_v})D \bigg) \bigg( \tilde{A}_u (BD)^{u-1}\bigg\{d+1+\frac{1}{BD-1}\bigg\}+\tilde{b}_u \bigg) \\
    & \quad \quad \quad \quad \quad \quad \quad \quad \quad \quad \quad \quad \quad \quad \quad \quad \quad \quad \quad \quad \quad \quad \quad \quad + \bigg( \prod_{v=2}^\ell (B+\tilde{A_v})D \bigg) r_1(\theta) \\
    & \leq \sum_{u=2}^\ell D^{\ell-u} \prod_{v=u+1}^\ell (B+\tilde{A_v}) \tilde{A}_u(BD)^{u-1}\bigg(d+1+\frac{1}{BD-1}\bigg)  \\
    & \quad \quad \quad \quad \quad \quad \quad \quad \quad \quad \quad + \sum_{u=2}^\ell D^{\ell-u} \prod_{v=u+1}^\ell (B+\tilde{A_v}) \tilde{b}_u + D^{\ell-1} \prod_{v=2}^\ell (B+\tilde{A_v}) r_1(\theta),
\end{align*}
and using inequality $r_1(\theta)\leq d \cdot \tilde{A}_1+\tilde{b}_1$:
    \begin{align*}
    r_\ell(\theta) & \leq D^{\ell-1} \bigg(d+1+\frac{1}{BD-1}\bigg) \sum_{u=2}^\ell B^{u-1} \prod_{v=u+1}^\ell (B+\tilde{A}_v)\tilde{A}_u + \sum_{u=2}^\ell D^{\ell-u} \prod_{v=u+1}^\ell (B+\tilde{A}_v)\tilde{b}_u \\
    & \quad \quad \quad \quad \quad \quad \quad \quad \quad \quad \quad \quad \quad \quad + dD^{\ell-1} \prod_{v=2}^\ell (B+\tilde{A}_v) \tilde{A}_1 + D^{\ell-1} \prod_{v=2}^\ell (B+\tilde{A}_v) \tilde{b}_1 \\
    & \leq D^{\ell-1} \bigg(d+1+\frac{1}{BD-1}\bigg) \sum_{u=1}^\ell B^{u-1} \prod_{v=u+1}^\ell (B+\tilde{A}_v)\tilde{A}_u + \sum_{u=1}^\ell D^{\ell-u} \prod_{v=u+1}^\ell (B+\tilde{A}_v)\tilde{b}_u .
\end{align*}
Then we have for any distribution $q(\theta) = q_1(\theta_1)\times...\times q_T(\theta_T)$:
\begin{align*}
    \int & \|f_{\theta}-f_{\theta^*}\|_2^2 q({\rm d} \theta) \leq \int \|f_{\theta}-f_{\theta^*}\|_{\infty}^2 q({\rm d} \theta) = \int r_L(\theta)^2 q({\rm d} \theta) \\
    & \leq \int 2 D^{2L-2} \bigg(d+1+\frac{1}{BD-1}\bigg)^2 \bigg( \sum_{\ell=1}^L B^{\ell-1} \prod_{v=\ell+1}^L (B+\tilde{A}_v)\tilde{A}_\ell \bigg)^2 q({\rm d} \theta) \\
    & \quad \quad \quad \quad \quad \quad \quad \quad \quad \quad \quad \quad + \int 2 \bigg( \sum_{\ell=1}^L D^{L-\ell} \prod_{v=\ell+1}^L (B+\tilde{A}_v)\tilde{b}_\ell \bigg)^2 q({\rm d} \theta) \\
    & = 2 D^{2L-2}\bigg(d+1+\frac{1}{BD-1}\bigg)^2 \bigg( \int \sum_{\ell=1}^L B^{2\ell-2} \prod_{v=\ell+1}^L (B+\tilde{A}_v)^2 \tilde{A}_\ell^2 q({\rm d} \theta) \\
    & \quad \quad \quad \quad \quad \quad \quad \quad \quad \quad + 2 \int \sum_{\ell=1}^L \sum_{k=1}^{\ell-1} B^{\ell-1} B^{k-1} \prod_{v=\ell+1}^L (B+\tilde{A}_v)\tilde{A}_\ell \prod_{v=k+1}^L (B+\tilde{A}_v)\tilde{A}_k q({\rm d} \theta) \bigg) \\
    & \quad + 2 \bigg( \int \sum_{\ell=1}^L D^{2(L-\ell)} \prod_{v=\ell+1}^L (B+\tilde{A}_v)^2 \tilde{b}_\ell^2 q({\rm d} \theta) \\
    & \quad \quad \quad \quad \quad \quad \quad \quad + 2 \int \sum_{\ell=1}^L \sum_{k=1}^{\ell-1} D^{L-\ell} D^{L-k} \prod_{v=\ell+1}^L (B+\tilde{A}_v)\tilde{b}_\ell \prod_{v=k+1}^L (B+\tilde{A}_v)\tilde{b}_k q({\rm d} \theta) \bigg) \\
    & = 2 D^{2L-2}\bigg(d+1+\frac{1}{BD-1}\bigg)^2 \bigg( \sum_{\ell=1}^L B^{2\ell-2} \prod_{v=\ell+1}^L \int (B+\tilde{A}_v)^2 q({\rm d} \theta) \int \tilde{A}_\ell^2 q_\ell({\rm d} \theta_\ell) \\
    & \quad + 2 \sum_{\ell=1}^L \sum_{k=1}^{\ell-1} B^{\ell-1} B^{k-1} \prod_{v=\ell+1}^L \int (B+\tilde{A}_v)^2 q({\rm d} \theta) \int \tilde{A}_\ell q_\ell({\rm d} \theta_\ell) \prod_{v=k+1}^{\ell} \int (B+\tilde{A}_v) q({\rm d} \theta) \int \tilde{A}_k q({\rm d} \theta) \bigg) \\
    & \quad + 2 \bigg( \sum_{\ell=1}^L D^{2(L-\ell)} \prod_{v=\ell+1}^L \int (B+\tilde{A}_v)^2 q({\rm d} \theta) \int \tilde{b}_\ell^2 q({\rm d} \theta) \\
    & \quad + 2 \sum_{\ell=1}^L \sum_{k=1}^{\ell-1} D^{L-\ell} D^{L-k} \prod_{v=\ell+1}^L \int (B+\tilde{A}_v)^2 q({\rm d} \theta) \int \tilde{b}_\ell q({\rm d} \theta) \prod_{v=k+1}^{\ell} \int (B+\tilde{A}_v) q({\rm d} \theta) \int \tilde{b}_k q({\rm d} \theta) \bigg).
\end{align*}

Here, we define $q_n^*(\theta)$ as follows:
$$
\begin{cases}
\gamma_t^* = \mathbb{I}(\theta^*_t\ne 0), \\
\theta_t \sim \gamma_t^* \hspace{0.1cm} \mathcal{N}(\theta_t^*,s_n^2) + (1-\gamma_t^*)\delta_{\{0\}} \hspace{0.3cm} \textnormal{for each} \hspace{0.2cm} t=1,...,T.
\end{cases}
$$
with $s_n^2 = \frac{S}{16n} \log(3D)^{-1} (2BD)^{-2L} \bigg\{ \bigg(d+1+\frac{1}{BD-1} \bigg)^2 + \frac{1}{(2BD)^2-1} + \frac{2}{(2BD-1)^2} \bigg\}^{-1}$. 

We upper bound the expectation of the supremum of absolute values of Gaussian variables:
$$
\int \tilde{A}_\ell q_n^*({\rm d} \theta) = \int \sup_{i,j} |A_{\ell,i,j}-A_{\ell,i,j}^*| q_n^*({\rm d} \theta) \leq \sqrt{2s_n^2\log(2D^2)} = \sqrt{4s_n^2\log (3D)} ,
$$
and use Example 2.7 in \cite{boucheron2003}:
$$
\int \tilde{A}_\ell^2 q_n^*({\rm d} \theta) = \int \sup_{i,j} (A_{\ell,i,j}-A_{\ell,i,j}^*)^2 q_n^*({\rm d} \theta) \leq s_n^2 (1+2\sqrt{\log(D^2)}+\log(D^2)) = 4 s_n^2 \log(3D) ,
$$
which also give:
$$
\int (B+\tilde{A}_\ell) q_n^*({\rm d} \theta) = B + \int \tilde{A}_\ell q_n^*({\rm d} \theta) \leq B+\sqrt{4s_n^2\log (3D)} \leq 2B ,
$$
and
\begin{align*}
    \int (B+\tilde{A}_\ell)^2 q_n^*({\rm d} \theta) & = B^2 + 2B \int \tilde{A}_\ell q_n^*({\rm d} \theta) + \int \tilde{A}_\ell^2 q_n^*({\rm d} \theta) \\
    & \leq B^2 + 2B \sqrt{4s_n^2\log (3D)} + 4 s_n^2 \log (3D) \\
    & \leq 4B^2 
\end{align*}
as $\sqrt{4s_n^2\log (3D)} \leq B$ ($s_n^2 \leq \frac{LD(D+1)}{16n} (2BD)^{-2L} \leq \frac{2LD^2}{16n} 4^{-2L} D^{-2L} \leq 1$).

\vspace{0.2cm}

Similarly,
$$
\int \tilde{b}_\ell q_n^*({\rm d} \theta) \leq \sqrt{4s_n^2\log (3D)} 
$$
and
$$
\int \tilde{b}_\ell^2 q_n^*({\rm d} \theta) \leq 4 s_n^2 \log (3D) .
$$

Then
\begin{align*}
    \int \|f_{\theta}-f_{\theta^*}\|_2^2 q_n^*({\rm d} \theta)
    & \leq 2 D^{2L-2}\bigg(d+1+\frac{1}{BD-1}\bigg)^2 \bigg( \sum_{\ell=1}^L B^{2\ell-2} (4B^2)^{L-\ell} 4s_n^2\log (3D) \\
    & \quad + 2 \sum_{\ell=1}^L \sum_{k=1}^{\ell-1} B^{\ell-1} B^{k-1} (4B^2)^{L-\ell} \sqrt{4s_n^2\log (3D)} (2B)^{\ell-k} \sqrt{4s_n^2\log (3D)} \bigg) \\
    & \quad + 2 \bigg( \sum_{\ell=1}^L D^{2(L-\ell)} (4B^2)^{L-\ell} 4s_n^2\log (3D) \\
    & \quad + 2 \sum_{\ell=1}^L \sum_{k=1}^{\ell-1} D^{L-\ell} D^{L-k} (4B^2)^{L-\ell} \sqrt{4s_n^2\log (3D)} (2B)^{\ell-k} \sqrt{4s_n^2\log (3D)} \bigg) ,
\end{align*}
i.e.
\begin{align*}
    \int & \|f_{\theta}-f_{\theta^*}\|_2^2 q_n^*({\rm d} \theta) \\
    & \leq 2 D^{2L-2}\bigg(d+1+\frac{1}{BD-1}\bigg)^2 \bigg( B^{2L-2} 4s_n^2\log (3D) \sum_{\ell=0}^{L-1} 4^\ell \\
    & \quad \quad \quad \quad \quad \quad \quad \quad \quad \quad \quad \quad + 2 B^{2L-2} 4s_n^2\log (3D) \sum_{\ell=1}^L \sum_{k=1}^{\ell-1} 2^{L-\ell} 2^{L-k} \bigg) \\
    & \quad + 2 \bigg( 4s_n^2\log (3D) \sum_{\ell=1}^L (2BD)^{2L-2\ell} + 8s_n^2\log (3D) \sum_{\ell=1}^L \sum_{k=1}^{\ell-1} (2BD)^{L-\ell} (2BD)^{L-k} \bigg) \\
    & \leq 2 D^{2L-2}\bigg(d+1+\frac{1}{BD-1}\bigg)^2 \bigg( B^{2L-2} 4s_n^2\log (3D) \frac{4^L-1}{4-1} \\
    & \quad \quad \quad \quad \quad \quad \quad \quad \quad \quad \quad \quad + 2 B^{2L-2} 4s_n^2\log (3D) \sum_{\ell=1}^L 2^{L-\ell} 2^{L-\ell+1} \sum_{k=0}^{\ell-2} 2^k \bigg) \\
    & \quad + 2 \bigg( 4s_n^2\log (3D) \sum_{\ell=0}^{L-1} (2BD)^{2\ell} + 8s_n^2\log (3D) \sum_{\ell=1}^L (2BD)^{L-\ell} (2BD)^{L-\ell+1} \sum_{k=0}^{\ell-2} (2BD)^k \bigg) \\
    & \leq 2 D^{2L-2}\bigg(d+1+\frac{1}{BD-1}\bigg)^2 \bigg( B^{2L-2} 4s_n^2\log (3D) \frac{4^L}{3} \\
    & \quad \quad \quad \quad \quad \quad \quad \quad \quad \quad \quad \quad + 2 B^{2L-2} 4s_n^2\log (3D) \sum_{\ell=1}^L 2^{L-\ell} 2^{L-\ell+1} 2^{\ell-1} \bigg) \\
    & \quad + 2 \bigg( 4s_n^2\log (3D) \frac{(2BD)^{2L}}{(2BD)^2-1} + 8s_n^2\log (3D) \sum_{\ell=1}^L (2BD)^{L-\ell} (2BD)^{L-\ell+1} \frac{(2BD)^{\ell-1}}{2BD-1} \bigg) \\
    & \leq 2 D^{2L-2}\bigg(d+1+\frac{1}{BD-1}\bigg)^2 \bigg( B^{2L-2} 4s_n^2\log (3D) \frac{4^L}{3} + 2 B^{2L-2} 4s_n^2\log (3D) 2^L \sum_{\ell=0}^{L-1} 2^\ell \bigg) \\
    & \quad + 2 \bigg( 4s_n^2\log (3D) \frac{(2BD)^{2L}}{(2BD)^2-1} + 8s_n^2\log (3D) \sum_{\ell=0}^{L-1} (2BD)^\ell \frac{(2BD)^L}{2BD-1} \bigg) \\
    & \leq 2 D^{2L-2}\bigg(d+1+\frac{1}{BD-1}\bigg)^2 \bigg( B^{2L-2} 4s_n^2\log (3D) \frac{4^L}{3} + 2 B^{2L-2} 4s_n^2\log (3D) 2^{2L} \bigg) \\
    & \quad + 2 \bigg( 4s_n^2\log (3D) \frac{(2BD)^{2L}}{(2BD)^2-1} + 8s_n^2\log (3D) \frac{(2BD)^{2L}}{(2BD-1)^2} \bigg) \\
    & = 2 D^{2L-2}\bigg(d+1+\frac{1}{BD-1} \bigg)^2 4s_n^2\log(3D) \bigg( B^{2L-2} \frac{4^L}{3} + 2 B^{2L-2} 2^{2L} \bigg) \\
    & \quad + 2 \bigg( \frac{(2BD)^{2L}}{(2BD)^2-1} + 2 \frac{(2BD)^{2L}}{(2BD-1)^2} \bigg) 4s_n^2\log(3D) ,
\end{align*}
and consequently, as $BD\geq2$,
\begin{align*}
    \int & \|f_{\theta}-f_{\theta^*}\|_2^2 q_n^*({\rm d} \theta) \\
    & \leq 8 s_n^2\log(3D) \bigg\{ D^{2L-2} \bigg(d+1+\frac{1}{BD-1} \bigg)^2 \frac{7}{3} B^{2L-2} 2^{2L} \\
    & \quad \quad \quad \quad \quad \quad \quad \quad \quad \quad \quad \quad \quad \quad \quad + (2BD)^{2L} \bigg( \frac{1}{(2BD)^2-1} + \frac{2}{(2BD-1)^2} \bigg) \bigg\} \\
    & = 8 s_n^2\log(3D) \bigg\{ (2BD)^{2L} \frac{1}{(BD)^2} \bigg(d+1+\frac{1}{BD-1} \bigg)^2 \frac{7}{3} \\
    & \quad \quad \quad \quad \quad \quad \quad \quad \quad \quad \quad \quad \quad \quad \quad + (2BD)^{2L} \bigg( \frac{1}{(2BD)^2-1} + \frac{2}{(2BD-1)^2} \bigg) \bigg\} \\
    & \leq 8 s_n^2\log(3D) (2BD)^{2L} \bigg\{ \bigg(d+1+\frac{1}{BD-1} \bigg)^2 + \frac{1}{(2BD)^2-1} + \frac{2}{(2BD-1)^2} \bigg\} \\
    & = \frac{S}{2n} \\
    & \leq r_n .
\end{align*}
which ends Step 2.

\vspace{0.3cm}

\underline{\textit{Third step}}: \underline{\textit{we prove Inequality \eqref{KLcond2Gauss}}}

\vspace{0.3cm}

We end the proof:
\begin{align*}
    \textnormal{KL}(q_n^*\|\pi) & \leq \log\binom{T}{S} + \sum_{t=1}^{T} \gamma_t^* \textnormal{KL}\bigg( \mathcal{N}(\theta_t^*,s_n^2) \bigg\| \mathcal{N}(0,1) \bigg) \\
    & \leq S\log(T) + \sum_{t=1}^{T} \gamma_t^* \bigg\{ \frac{1}{2} \log\bigg(\frac{1}{s_n^2}\bigg) +\frac{s_n^2+\theta_t^{*2}}{2}-\frac{1}{2} \bigg\} \\
    & \leq S\log(T) + \sum_{t=1}^{T} \gamma_t^* \bigg\{ \frac{1}{2} \log\bigg(\frac{1}{s_n^2}\bigg) +\frac{s_n^2+B^2}{2}-\frac{1}{2} \bigg\} \\
    & = S\log(T) + \frac{S}{2}s_n^2 + \frac{S}{2} \frac{B^2-1}{2} + \frac{S}{2} \log\bigg(\frac{1}{s_n^2}\bigg) \\
    & \leq S\log(T) + \frac{S}{2} + \frac{S}{2} \frac{B^2-1}{2} \\
    & \quad + \frac{S}{2} \log\bigg(\frac{16n}{S} \log(3D) (2BD)^{2L} \bigg\{ \bigg(d+1+\frac{1}{BD-1} \bigg)^2 + \frac{1}{(2BD)^2-1} + \frac{2}{(2BD-1)^2} \bigg\}\bigg) \\
    & \leq S\log(L(D+1)^2) + \frac{B^2 S}{4} + LS\log(2BD) + \frac{S}{2} \log\log (3D)  \\
    & \quad + \frac{S}{2} \log\bigg(\frac{16n}{S}\bigg\{ \bigg(d+1+\frac{1}{BD-1}\bigg)^2 + \frac{1}{(BD)^2-1} + \frac{2}{(BD-1)^2} \bigg\}\bigg) \\
    & \leq n r_n .
\end{align*}

\vspace{-0.2cm}

\end{proof}

\end{document}